\documentclass[reqno, oneside, 12pt]{amsart}

\usepackage[letterpaper]{geometry}
\usepackage{amsmath}
\usepackage{amssymb}
\usepackage{enumerate}
\usepackage{verbatim}
\usepackage{mathrsfs}
\usepackage{mathtools}

\mathtoolsset{centercolon}

\usepackage{setspace}
\numberwithin{equation}{section}

\newtheorem{thm}{Theorem}[section]
\newtheorem{cor}[thm]{Corollary}
\newtheorem{lem}[thm]{Lemma}
\newtheorem{prop}[thm]{Proposition} 

\newtheorem{conjec}[thm]{Conjecture}

\newtheorem{defn}[thm]{Definition}

\theoremstyle{remark}
\newtheorem*{rem*}{Remark}

\date{}

\author{Heidi Goodson}
\address{Department of Mathematics, University of Minnesota}
\email{goods052@umn.edu}

\title[Hypergeometric Functions and Dwork Hypersurfaces]{Hypergeometric Functions and Relations to Dwork Hypersurfaces}

\begin{document}

\begin{abstract}
We give an expression for number of points for the family of Dwork K3 surfaces
$$X_{\lambda}^4: \hspace{.1in} x_1^4+x_2^4+x_3^4+x_4^4=4\lambda x_1x_2x_3x_4$$
over finite fields of order $q\equiv 1\pmod 4$ in terms of Greene's finite field hypergeometric functions. We also develop hypergeometric point count formulas for all odd primes using McCarthy's $p$-adic hypergeometric function. Furthermore, we investigate the relationship between certain period integrals of these surfaces and the trace of Frobenius over finite fields. We extend this work to higher dimensional Dwork hypersurfaces.
\end{abstract}

\maketitle

\section{Introduction}
The motivation for this work comes from a particular family of elliptic curves. For $\lambda\not= 0,1$ we define an elliptic curve in the Legendre family by 
$$E_\lambda : y^2 = x(x-1)(x-\lambda).$$

We compute a period integral associated to the Legendre elliptic curve given by integrating the nowhere vanishing holomorphic $1$-form $\omega=\frac{dx}{y}$ over a $1$-dimensional cycle containing $\lambda$. This period is a solution to a hypergeometric differential equation and can be expressed as the classical hypergeometric series
$$\pi= \int_{0}^{\lambda} \frac{dx}{y}={}_2F_1\left(\left.\begin{array}{cc}
                \frac12&\frac12	\\
		{}&1
               \end{array}\right|\lambda\right).$$
See the exposition in \cite{Clemens} for more details on this. \\

We now specialize to the case where $\lambda\in\mathbb Q\setminus\{0,1\}$. Koike \cite[Section 4]{Koike1995} showed that, for all odd primes $p$, the trace of Frobenius for curves in this family can be expressed in terms of Greene's hypergeometric function 
$$a_{E_\lambda}(p)=-\phi(-1)p\cdot {}_2F_1\left(\left.\begin{array}{cc}
                \phi&\phi\\
		{}&\epsilon
               \end{array}\right|\lambda\right)_{p},  $$
where $\epsilon$ is the trivial character and $\phi$ is a quadratic character modulo $p$.\\ 

Note the similarity between the period and trace of Frobenius expressions: the period is given by a classical hypergeometric series whose arguments are the fractions with denominator 2 and the trace of Frobenius is given by a finite field hypergeometric function whose arguments are characters of order 2. This similarity is to be expected for curves. Manin proved in \cite{Manin} that the rows of the Hasse-Witt matrix of an algebraic curve are solutions to the differential equations of the periods. In the case where the genus is 1, the Hasse-Witt matrix has a single entry: the trace of Frobenius.  Igusa showed in \cite{Igusa1958} that the trace of Frobenius is congruent modulo $p$ to the classical hypergeometric expression 
\begin{equation*}(-1)^{\frac{p-1}{2}}{}_2F_1\left(\left.\begin{array}{cc}
                \frac12&\frac12	\\
{}&1
               \end{array}\right|\lambda\right)
\end{equation*}                           
for odd primes $p$ (see the exposition in Clemens' book \cite{Clemens}). Furthermore, in Corollary \ref{cor:2F1ECcongruence} we show that these classical and finite field ${}_2F_1$ hypergeometric expressions are congruent modulo $p$ for odd primes. This result would imply merely a congruence between the finite field hypergeometric function expression and the point count over $\mathbb F_p$. The fact that Koike showed that we actually have an equality is very intriguing and leads us to wonder for what other varieties this type of equality holds.\\

Further examples of this correspondence have been observed for algebraic curves \cite{Swisher2015Arxiv, Fuselier10, Lennon1, Mortenson2003a} and for particular Calabi-Yau threefolds \cite{AhlgrenOno00a, McCarthy2012b}. For example, Fuselier \cite{Fuselier10} gave a finite field hypergeometric trace of Frobenius formula for elliptic curves with $j$-invariant $\frac{1728}{t}$, where $t\in \mathbb F_p \setminus \{0,1\}$.  Lennon \cite{Lennon1} extended this by giving a hypergeometric trace of Frobenius formula that does not depend on the Weierstrass model chosen for the elliptic curve. In \cite{AhlgrenOno00a}, Ahlgren and Ono gave a formula for the number of $\mathbb F_p$ points on a modular Calabi-Yau threefold. We extend these works to Dwork hypersurfaces, largely focusing on results that hold for Dwork K3 surfaces. Recall that the family of Dwork K3 surfaces is defined by
$$X_{\lambda}^4: \hspace{.1in} x_1^4+x_2^4+x_3^4+x_4^4=4\lambda x_1x_2x_3x_4.$$
We show that the number of points on the family of Dwork K3 surfaces over finite fields can be expressed in terms of Greene's finite field hypergeometric functions. The following is proved in Section \ref{sec:DworkSurfaces}.

\begin{thm}\label{thm:K3PointCount}
Let $q=p^e$ be a prime power such that $q\equiv 1\pmod 4$, $t=\frac{q-1}{4}$, and $T$ be a generator for $\widehat{\mathbb F_q^{\times}}$. When $\lambda^4=1$ we have
 $$\#X_{\lambda}^4(\mathbb F_q)=\frac{q^3-1}{q-1}+3qT^t(-1)+q^2{}_{3}F_{2}\left(\left.\begin{array}{ccc}
                T^t&T^{2t}&T^{3t}\\
		{} &\epsilon&\epsilon
               \end{array}\right|1\right)_q.$$               
More generally, for $\lambda\not=0$,
\begin{align*}
 \#X_{\lambda}^4(\mathbb F_q)&=\frac{q^3-1}{q-1}+12qT^t(-1)T^{2t}(1-\lambda^4)\\
               &\hspace{.2in}+q^2{}_{3}F_{2}\left(\left.\begin{array}{ccc}
                T^t&T^{2t}&T^{3t}\\
{} &\epsilon&\epsilon
               \end{array}\right|\frac{1}{\lambda^4}\right)_q+3q^2\binom{T^{3t}}{T^t}{}_{2}F_{1}\left(\left.\begin{array}{cc}
                T^{3t}&T^{t}\\
{} &T^{2t}
               \end{array}\right|\frac{1}{\lambda^4}\right)_q, 
\end{align*}
for all prime powers $q\equiv 1\pmod 4$ away from the primes of $\lambda$.

\end{thm}

\begin{rem*}
Salerno \cite{Salerno2013} used a finite field hypergeometric function defined by Katz \cite{Katz1990} to develop a point count formula for a larger class of diagonal hypersurfaces that could specialize to Dwork K3 surfaces. Theorem 5.5 of Salerno's paper gives a congruence between the number of points on diagonal surfaces and classical truncated hypergeometric series. This specializes to a single hypergeometric term in the Dwork K3 surface case (see Section 5.4 of Salerno's paper). Our result in Theorem \ref{thm:K3PointCount} gives an exact formula for the point count, not just a congruence. Furthermore, in Section \ref{sec:HGFCongruence} of this paper we give a congruence between the ${}_3F_2$ finite field hypergeometric function of Theorem \ref{thm:K3PointCount} and classical truncated hypergeometric series that appears in Section 5.4 of Salerno's paper. We later use this congruence to prove a result that gives the relationship between the trace of Frobenius and certain periods associated to Dwork K3 surfaces, so for our purposes, Greene's finite field hypergeometric function is a natural choice of functions to work with. 
%We note that Greene's hypergeometric functions were also used for point count, trace of Frobenius, and Fourier coefficients of modular forms results in the work of \cite{AhlgrenOno00a,Swisher2015Arxiv,FOP1, FOP2,Fuselier10,Kilbourn2006,Koike1995,Lennon1, Lennon2,Mortenson2003a,Mortenson2005}.\\
 
\end{rem*}

Note that the character $T^t$ is only defined over $\mathbb F_q$ when $q\equiv 1\pmod 4$. We would like to develop a point count formula to use for fields $\mathbb F_p$, where $p\equiv 3\pmod 4$ is prime. In this case, it seems unlikely that we will be able to develop a finite field hypergeometric formula. However, McCarthy \cite{McCarthy2013} defined a $p$-adic version of these hypergeometric functions which we will use to write a concise formula to calculate the number of points on Dwork K3 surfaces over these fields.
\begin{thm}\label{thm:K3PointCountnGn}
 When $p\equiv 3 \pmod 4$ and $\lambda\not=0$, the point count is given by
\begin{align*}
\#X_{\lambda}^4(\mathbb F_p)&=\frac{p^3-1}{p-1}+{}_3G_3\left[\left.\begin{array}{ccc}
                1/4&2/4&3/4\\
		0&0&0
               \end{array}\right|{\lambda^4}\right]_p -3p\hspace{.02in}{}_2G_2\left[\left.\begin{array}{ccc}
                3/4&1/4\\
		0&1/2
               \end{array}\right|{\lambda^4}\right]_p,
\end{align*}
for all primes $p\equiv 1\pmod 4$ away from the primes of $\lambda$.
\end{thm}

Though we have already developed a hypergeometric point count formula that holds for primes $p\equiv 1 \pmod 4$ in Theorem \ref{thm:K3PointCount}, we show that we can also give a $p$-adic hypergeometric formula for these primes. Note the similarity between this formula and that of Theorem \ref{thm:K3PointCountnGn}.

\begin{thm}\label{thm:K3PointCountnGn1}
 When $p\equiv 1 \pmod 4$  and $\lambda\not=0$, the point count is given by
\begin{align*}
\#X_{\lambda}^4(\mathbb F_p)&=\frac{p^3-1}{p-1}+12pT^t(-1)T^{2t}(1-\lambda^4)\\
&\hspace{1in}+{}_3G_3\left[\left.\begin{array}{ccc}
                1/4&2/4&3/4\\
		0&0&0
               \end{array}\right|{\lambda^4}\right]_p+3p{}_2G_2\left[\left.\begin{array}{cc}
                3/4&1/4\\
		0&2/4
               \end{array}\right|\lambda^4\right]_p,
\end{align*}
for all primes $p\equiv 1\pmod 4$ away from the primes of $\lambda$.
\end{thm}

We observe an interesting phenomenon with certain periods associated to the Dwork surfaces we have studied. For an $n$-dimensional Dwork hypersurface, we calculate a period integral, obtained by choosing dual bases of the space of holomorphic $(n,0)$-differentials and the space of cycles ${H^n(X_\lambda^{n+2},\mathscr O)}$ and integrating the differentials over each cycle. Note that the dimension of both spaces is 1 since Dwork hypersurfaces are Calabi-Yau manifolds and, therefore, have genus $g=1$. The natural choice for a basis of differentials would be the nowhere vanishing holomorphic $n$-form.  These periods can be written in terms of classical hypergeometric series, a fact that was first noted by Dwork in \cite{Dwork1969}. Interestingly, the hypergeometric expressions for the periods and the point counts "match" in the sense that fractions with denominator $a$ in the classical series coincide with characters of order $a$ in the finite field hypergeometric functions. We use these matching expressions to prove the following result for Dwork K3 surfaces in Section \ref{sec:PeriodTrace}.

\begin{thm}\label{thm:K3PeriodTrace}
 For the Dwork K3 surface
 $$X_{\lambda}^4:\hspace{.1in} x_1^4+x_2^4+x_3^4+x_4^4=4\lambda x_1x_2x_3x_4,$$
 the trace of Frobenius over $\mathbb F_p$ and the period associated to the surface are congruent modulo $p$ when $p \equiv 1 \pmod 4$.
\end{thm}
We conjecture an analogous result for higher dimensional Dwork hypersurfaces
$$X_{\lambda}^d:\hspace{.1in} x_1^d+x_2^d+\ldots+x_d^d=d\lambda x_1x_2\cdots x_d$$
in Section \ref{sec:DworkHypersurface}.\\

The remainder of this paper is organized as follows. In Section \ref{sec:background} we give some necessary background information that will be used throughout the paper. Section \ref{sec:HGFCongruence} gives some congruences between classical hypergeometric series and finite field hypergeometric functions that we will use in our period and trace of Frobenius results. We prove Theorem \ref{thm:K3PointCount} in Section \ref{sec:DworkSurfaces}, using results that are proved in Section \ref{sec:Koblitz}. In Section \ref{sec:K3PointCountnGn} we prove the $p$-adic point count formulas of Theorems \ref{thm:K3PointCountnGn} and \ref{thm:K3PointCountnGn1}. We prove Theorem \ref{thm:K3PeriodTrace} in Section \ref{sec:PeriodTrace}. Finally, we extend these Dwork K3 surface results to higher dimensional hypersurfaces in Section \ref{sec:DworkHypersurface}.

\section*{Acknowledgements}

The author would like to thank her thesis advisor, Benjamin Brubaker, as well as Rupam Barman, Frits Beukers, Dermot McCarthy, and Steven Sperber for helpful conversations while working on these results.

\section{Preliminaries}\label{sec:background}
\subsection{Hypergeometric Series and Functions}
\label{sec:HGF}
We start by recalling the definition of the classical hypergeometric series
\begin{equation}\label{eqn:classicalHGF}
 {}_{n+1}F_{n}\left(\left.\begin{array}{cccc}
                a_0&a_1&\ldots& a_n\\
		{}&b_1&\ldots,& b_n
               \end{array}\right|x\right) = \displaystyle\sum_{k=0}^{\infty}\dfrac{(a_0)_k\ldots(a_n)_k}{(b_1)_k\ldots(b_n)_kk!}x^k,
\end{equation}
where $(a)_0=1$ and $(a)_k=a(a+1)(a+2)\ldots(a+k-1)$.  In Section \ref{sec:HGFCongruence} we will also be interested in truncated hypergeometric series. For a positive integer, $m$, define the hypergeometric series truncated at $m$ to be 

\begin{equation}\label{eqn:truncHGF}
 {}_{n+1}F_{n}\left(\left.\begin{array}{cccc}
                a_0&a_1&\ldots& a_n\\
		{}&b_1&\ldots,& b_n
               \end{array}\right|x\right)_{\text{tr}(m)} = \displaystyle\sum_{k=0}^{m-1}\dfrac{(a_0)_k\ldots(a_n)_k}{k!(b_1)_k\ldots(b_n)_k}x^k.
\end{equation}

In his 1987 paper \cite{Greene}, Greene introduced a finite field, character sum analogue of classical hypergeometric series that satisfies similar summation and transformation properties. Let $\mathbb F_q$ be the finite field with $q$ elements, where $q$ is a power of an odd prime $p$. If $\chi$ is a multiplicative character of $\widehat{\mathbb F_q^{\times}}$, extend it to all of $\mathbb F_q$ by setting $\chi(0)=0$. For any two characters $A,B$ of $\widehat{\mathbb F_q^{\times}}$ we define the normalized Jacobi sum by

\begin{equation}\label{eqn:normalizedjacobi}
 \binom{A}{B}:=\frac{B(-1)}{q}\sum_{x\in\mathbb F_q} A(x)\overline B(1-x) = \frac{B(-1)}{q}J(A,\bar{B}),
\end{equation}
where $J(A,B)=\sum_{x\in \mathbb F_q} A(x)B(1-x)$ is the usual Jacobi sum.\\

For any positive integer $n$ and characters $A_0,\ldots, A_n,B_1,\ldots, B_n$ in $\widehat{\mathbb F_q^{\times}}$, Greene defined the finite field hypergeometric function ${}_{n+1}F_n$ over $\mathbb F_q$ by
\begin{equation}\label{eqn:HGFdef}
 {}_{n+1}F_{n}\left(\left.\begin{array}{cccc}
                A_0,&A_1,&\ldots,&A_n\\
		{} &B_1,&\ldots,&B_n
               \end{array}\right|x\right)_q = \displaystyle\frac{q}{q-1}\sum_{\chi}\binom{A_0\chi}{\chi}\binom{A_1\chi}{B_1\chi}\ldots\binom{A_n\chi}{B_n\chi}\chi(x).
\end{equation}

In the case where $n=1$, an alternate definition, which is in fact Greene's original definition, is given by
\begin{equation}\label{eqn:2F1def}
 {}_{2}F_{1}\left(\left.\begin{array}{cc}
                A&B\\
		{} &C
               \end{array}\right|x\right)_q = \epsilon(x)\frac{BC(-1)}{q}\sum_yB(y)\overline{B}C(1-y)\overline{A}(1-xy).
\end{equation}

\subsection{Gauss and Jacobi Sums}
\label{sec:GaussJacobi}
Unless otherwise stated, information in this section can be found in Ireland and Rosen's text \cite[Chapter 8]{Ireland}.\\

We define the standard trace map $\text{tr}:\mathbb F_q\rightarrow \mathbb F_p$ by 
$$\text{tr}(x)=x+x^p+\ldots + x^{p^{e-1}}.$$
Let $\pi\in\mathbb C_p$ be a fixed root of $x^{p-1}+p=0$ and let $\zeta_p$ be the unique $p^\text{th}$ root of unity in $\mathbb C_p$ such that $\zeta_p \equiv 1+\pi \pmod{\pi^2}$. Then for $\chi \in\widehat{\mathbb F_q^{\times}}$ we define the Gauss sum $g(\chi)$ to be
\begin{equation}\label{eqn:GaussSum}
 g(\chi):=\sum_{x\in\mathbb F_q} \chi(x)\theta(x),
\end{equation}
where we define the additive character $\theta$ by $\theta(x)=\zeta_p^{\text{tr}(x)}$. Note that if  $\chi$ is nontrivial then $g(\chi)g(\overline\chi)=\chi(-1)q$.\\ 

We have the following connection between Gauss sums and Jacobi sums. For non-trivial characters $\chi$ and $\psi$ on $\mathbb F_q$ whose product is also non-trivial,
$$J(\chi,\psi)=\frac{g(\chi)g(\psi)}{g(\chi\psi)}.$$
More generally, for non-trivial characters $\chi_1,\ldots,\chi_n$ on $\mathbb F_q$ whose product is also non-trivial,
$$J(\chi_1,\ldots,\chi_n)=\frac{g(\chi_1)\cdots g(\chi_n)}{g(\chi_1\cdots \chi_n)}.$$

Another important product formula is the Hasse-Davenport formula.

\begin{thm}\label{thm:HasseDavenport}\cite[Theorem 10.1]{Lang1990}
 Let $m$ be a positive integer and let $q$ be a prime power such that $q\equiv 1 \pmod m$. For characters $\chi,\psi\in\widehat{\mathbb F_q^{\times}}$ we have 
$$\prod_{i=0}^{m-1}g(\chi^i\psi)=-g(\psi^m)\psi^{-m}(m) \prod_{i=0}^{m-1}g(\chi^i).$$
\end{thm}

In Section \ref{sec:Koblitz} we will use the following specializations of this.
\begin{cor}\label{cor:HasseDavenport}
 $$g(T^{4j})=\frac{\prod_{i=0}^3 g(T^{it+j})}{qT^{-4j}(4)T^t(-1)g(T^{2t})}.$$
\end{cor}

\begin{proof}
 This follows from Theorem \ref{thm:HasseDavenport} using $m=4, \chi=T^t,$ and $\psi=T^j$.
\end{proof}

\begin{cor}\label{cor:HasseDavenportGeneral}
 More generally,
 $$g(T^{dj})=\frac{\prod_{i=0}^{d-1} g(T^{it+j})}{T^{-dj}(d)\prod_{i=1}^{d-1}g(T^{it})},$$
 where $q\equiv 1\pmod{d}$, $t=\frac{q-1}{d}$, and $T$ is a generator for $\widehat{\mathbb F_q^{\times}}$.
\end{cor}

Though there are other relations for Gauss sum expressions (see, for example, \cite{Evans1981, Yamamoto1966}), the Hasse-Davenport formula is our main tool for simplifying the expressions that appear in this paper's results. However, the following theorem of Helverson-Pasotto will be useful for rewriting an expression that appears in Section \ref{sec:Koblitz}.

\begin{thm}\label{thm:Helversen1978}\cite[Theorem 2]{Helversen1978}
 $$\frac{1}{q-1}\sum_{\chi}g(A\chi)g(B\overline\chi)g(C\chi)g(D\overline\chi) = \frac{g(AB)g(AD)g(BC)g(CD)}{g(ABCD)}+q(q-1)AC(-1)\delta(ABCD),$$
 where $\delta(ABCD)=1$ if $ABCD$ is the trivial character and $0$ otherwise. 
\end{thm}

The following is a new result about Gauss sums that is also helpful for simplifying formulas in Section \ref{sec:Koblitz}.

\begin{prop}\label{prop:gaussproduct}
Let $q$ be a prime power such that $q\equiv 1\pmod 4$, $t=\frac{q-1}{4}$, and $T$ be a generator for $\widehat{\mathbb F_q^{\times}}$. Let $a,b$ be multiples of $t$ satisfying $a+b=2t$. Then 
$$g(T^{2t})\displaystyle \sum_{j=0}^{q-2} g(T^{j+a})g(T^{-j+b})T^j(-1)T^{4j}(\lambda)=q(q-1)T^b(-1)T^{2t}(1-\lambda^4).$$
\end{prop}
% \begin{rem*}
%  Note that when $\lambda^4=1$ this expression is equal to 0.
% \end{rem*}

\begin{proof}
We start by using Equation \ref{eqn:GaussSum} to write
\begin{align*}
   \sum_{j=0}^{q-2} g(T^{j+a})g(T^{-j+b})T^j(-1)T^{4j}(\lambda) &= \sum_{j=0}^{q-2}T^j(-\lambda^4) \left(\sum_{x\in\mathbb F_q}T^{j+a}(x)\theta(x)\right) \left(\sum_{y\in\mathbb F_q}T^{-j+b}(y)\theta(y)\right)\\
   &= \sum_{j=0}^{q-2}T^j(-\lambda^4) \sum_{x,y\in\mathbb F_q}T^{j+a}(x)T^{-j+b}(y)\theta(x+y)\\ %can assume x,y\not=0 since there sum=0 anyways
   &= \sum_{j=0}^{q-2}T^j(-\lambda^4) \sum_{x,y\in\mathbb F_q^{\times}}T^{j}(x/y)T^{a}(x)T^{b}(y)\theta(x+y)\\
   &= \sum_{x,y\in\mathbb F_q^{\times}}T^{a}(x)T^{b}(y)\theta(x+y)\sum_{j=0}^{q-2}T^j\left(-\tfrac{\lambda^4x}{y}\right).
  \end{align*}
Note that $\sum_{j=0}^{q-2}T^j\left(-\tfrac{\lambda^4x}{y}\right)=0$ unless $-\frac{\lambda^4x}{y}=1$, in which case the sum equals $q-1$. So, we let $x=-\frac{y}{\lambda^4}$ to get
\begin{align*}
 \sum_{j=0}^{q-2} g(T^{j+a})g(T^{-j+b})T^j(-1)T^{4j}(\lambda)   &= (q-1)\sum_{y\in\mathbb F_q^{\times}}T^{a}\left(-\tfrac{y}{\lambda^4}\right)T^{b}(y)\theta\left(-\tfrac{y}{\lambda^4}+y\right)\\
 &= (q-1)\sum_{y\in\mathbb F_q^{\times}}T^{a}\left(-\tfrac{y}{\lambda^4}\right)T^{b}(y)\theta\left(y(-\tfrac{1}{\lambda^4}+1)\right).
  \end{align*}
  
We now consider two cases. First, suppose $\lambda^4=1$. Then we have 
\begin{equation*}
\sum_{y\in\mathbb F_q^{\times}}T^{a}(-y)T^{b}(y)\theta(0) %&=T^a(-1)\sum_{y\in\mathbb F_q^{\times}}T^{a+b}(y)\\
 =T^a(-1)\sum_{y\in\mathbb F_q^{\times}}T^{2t}(y)= 0
\end{equation*}
since $T^{2t}$ is not a trivial character. Now suppose $\lambda^4 \not=1$. Then we perform the change of variables $y\to y(-1/\lambda^4+1)^{-1}$ to get
\begin{align*}
\sum_{y\in\mathbb F_q^{\times}}T^{a}\left(\tfrac{-y}{\lambda^4-1}\right)T^{b}\left(\tfrac{y}{-1/\lambda^4+1}\right)\theta(y) &=T^{-a}(1-\lambda^4)T^{-b}\left(\tfrac{-1}{\lambda^4}+1\right)\sum_{y\in\mathbb F_q^{\times}}T^{a+b}(y)\theta(y)\\
 &=T^{b-2t}(1-\lambda^4)T^{-b}\left(\tfrac{\lambda^4-1}{\lambda^4}\right)\sum_{y\in\mathbb F_q^{\times}}T^{2t}(y)\theta(y)\\
 &=T^{b}(-\lambda^4)T^{-2t}(1-\lambda^4)g(T^{2t}).
\end{align*}
Note that if $\lambda^4=1$ then $T^{b}(-\lambda^4)T^{-2t}(1-\lambda^4)g(T^{2t})=0$, so we can use this expression for all $\lambda$. Hence, 
\begin{align*}
 g(T^{2t})\displaystyle \sum_{j=0}^{q-2} g(T^{j+a})g(T^{-j+b})T^j(-1)T^{4j}(\lambda)&= g(T^{2t})\cdot (q-1)T^{b}(-\lambda^4)T^{-2t}(1-\lambda^4)g(T^{2t})\\
 &=qT^{2t}(-1)\cdot(q-1)T^{b}(-\lambda^4)T^{-2t}(1-\lambda^4)\\
 %&=q(q-1)T^{b}(-\lambda^4)T^{2t}(1-\lambda^4)\\
 &=q(q-1)T^{b}(-1)T^{2t}(1-\lambda^4),
\end{align*}
where the last equation holds because $b$ is a multiple of $t$ and $T^{4t}(\lambda)=1$.

\end{proof}

\subsection{$p$-adic Gamma Function}\label{sec:padicGamma}
Throughout this section let $q=p^e$ be a power of an odd prime $p$ and let $\mathbb Z_p$ denote the ring of $p$-adic integers. Also, let $\mathbb Q_p$ denote the field of $p$-adic numbers and $\mathbb C_p$ be its $p$-adic completion. The following facts can be found in, for example, \cite{AhlgrenOno00a}, \cite{GrossKoblitz}, and \cite{Mortenson2003a}. We define the $p$-adic Gamma function $\Gamma_p: \mathbb Z_p \rightarrow \mathbb Z_p^*$ by
\begin{equation}\label{eqn:padicdefinition}
\Gamma_p(n):= (-1)^n \prod_{j<n, p\nmid j} j
\end{equation}
for numbers $n\in\mathbb N$. We extend this to all $x$ in $\mathbb Z_p$ by
$$\Gamma_p(x):= \lim_{n\rightarrow x}\Gamma_p(n),$$
where in the limit we take any sequence of positive integers that approaches $x$ $p$-adically. 

\begin{prop}\cite[Proposition 4.2]{Mortenson2003a}
\label{prop:Gamma_p}
If $p\geq5$ is prime, $x,y\in\mathbb Z_p$, and $z\in p\mathbb Z_p$, then 
\begin{enumerate}
 \item $\Gamma_p(x+1) =\left\{ 
  \begin{array}{l l}
    -x\Gamma_p(x) & \quad \text{if } x\in\mathbb Z_p^*,\\
    -\Gamma_p(x) & \quad \text{if } x\in p\mathbb Z_p.
  \end{array} \right.$
 \item If $n\geq1$ and $x\equiv y \pmod{p^n}$ then $\Gamma_p(x)\equiv\Gamma_p(y) \pmod{p^n}.$
 \item $\Gamma_p'(x+z)\equiv\Gamma_p'(x) \pmod p$.
 \item $|\Gamma_p(x)|=1$.
 \item Let $x_0\in \{1,\ldots, p\}$ be the constant term in the $p$-adic expansion of $x$. Then\\ $\Gamma_p(x)\Gamma_p(1-x)=(-1)^{x_0}.$
\end{enumerate}
\end{prop}

% We now define the logarithmic derivative of $\Gamma_p$. For $x\in\mathbb Z_p$, let
% $$G_1(x)=\frac{\Gamma_p'(x)}{\Gamma_p(x)}.$$ %\hspace{.5in} G_2(x)=\frac{\Gamma_p''(x)}{\Gamma_p(x)}.$$
% This function is defined on all of $\mathbb Z_p$ since $|\Gamma_p(x)|=1$ and $\Gamma_p$ is locally analytic. Using Proposition \ref{prop:Gamma_p} (1), we have 
% $$G_1(x+1)-G_1(x)=\frac1x$$
% for $x\in\mathbb Z_p$. Kilbourn \cite{Kilbourn2006} uses this logarithmic derivative in his proof of the level 8 supercongruence conjecture.
%Differentiating this gives us a similar identity for $G_2$:
%$$G_2(x+1)-G_2(x)=G_1(x+1)^2-G_1(x)^2-\frac{1}{x^2}.$$

The following proposition relates the $p$-adic Gamma function to the Pochhammer symbol.

\begin{prop}\label{prop:MortensonPadic}\cite[Proposition 5.1]{Mortenson2003a}
Let $m$ and $d$ be integers with $1\leq m<d$. If $p\equiv 1\pmod d$ is a prime, then define $t$ such that $t=\tfrac{p-1}{d}$. If $0\leq j \leq mt$, then
$$\frac{\Gamma_p\left(\tfrac{m}{d} +j\right)\Gamma_p\left(1-\tfrac{m}{d} +j\right)}{\Gamma_p\left(1+j\right)^2} = (-1)^{mt+1}\frac{\left(\tfrac{m}{d}\right)_j\left(\tfrac{d-m}{d}\right)_j}{j!^2}.$$

\end{prop}

We will use this proposition in the proof of Theorem \ref{thm:2F1congruence} in Section \ref{sec:HGFCongruence}.\\

We now state a relationship between Gauss sums and the $p$-adic Gamma function. Recall that $\pi\in\mathbb C_p$ is the fixed root of $x^{p-1}+p=0$ given in Section \ref{sec:GaussJacobi}. Define the Teichm\"uller character, $\omega$, to be the primitive character on $\mathbb F_p$ that is uniquely defined by the property $\omega(x)\equiv x \pmod p$ for all $x$ in $\{0,\ldots,p-1\}$. Then the Gross-Koblitz formula, specialized to the case where the field is $\mathbb F_p$, is the following.

\begin{thm}{\cite[Theorem 1.7]{GrossKoblitz}}\label{thm:GrossKoblitz}
 For $j$ in $\{0,\ldots,p-1\}$,
$$g(\overline{\omega}^j)=-\pi^j\Gamma_p\left(\frac{j}{p-1}\right).$$
\end{thm}

\subsection{$p$-adic Hypergeometric Function}
\label{sec:padicHGF}

In Section \ref{sec:HGF} we defined finite field hypergeometric functions, which will be used in may of the results in this thesis. These hypergeometric functions, however, are limited to primes in which the characters are defined. For example, in Theorem \ref{thm:K3PointCount} our point count result is limited to primes and prime powers congruent to 1 mod 4. In order to extend our results to primes in other congruence classes, we use a $p$-adic hypergeometric function developed by McCarthy in \cite{McCarthy2013}.\\

For $x\in \mathbb Q$ we let $\lfloor x\rfloor$ denote the greatest integer less than or equal to $x$ and let $\langle x \rangle$ denote the fractional part of $x$, i.e. $x-\lfloor x\rfloor $. 
 
\begin{defn}\label{def:padicHGF}\cite[Definition 1.1]{McCarthy2013}
Let $p$ be an odd prime and let $t\in\mathbb F_p$. For $n\in \mathbb Z^+$ and $1\leq i \leq n$, let $a_i,b_i\in \mathbb Q \cap \mathbb Z_p$. Then we define

\begin{align*}
 {}_nG_n&\left[\left.\begin{array}{cccc}
                a_1&a_2&\ldots&a_n\\
		b_1&b_2&\ldots&b_n
               \end{array}\right|t\right]_p :=\frac{-1}{p-1}\sum_{j=0}^{p-2}(-1)^{jn}\overline\omega^j(t)\\
               &\hspace{.2in}\times\prod_{i=1}^n\frac{\Gamma_p\left(\left\langle a_i-\frac{j}{p-1}\right\rangle\right)}{\Gamma_p\left(\langle a_i\rangle\right)}\frac{\Gamma_p\left(\left\langle -b_i+\frac{j}{p-1}\right\rangle\right)}{\Gamma_p\left(\langle -b_i\rangle\right)}(-p)^{-\lfloor\langle a_i\rangle-\frac{j}{p-1}\rfloor-\lfloor \langle-b_i\rangle +\frac{j}{p-1}\rfloor}.
\end{align*}

\end{defn}

\section{Hypergeometric Congruences}
\label{sec:HGFCongruence}

In this section we prove congruences between classical truncated hypergeometric series and finite field hypergeometric functions. This builds on results of Mortenson \cite{Mortenson2003a, Mortenson2005} by considering hypergeometric functions evaluated away from 1, though our results hold mod $p$ instead of $p^2$. \\

The first result is for ${}_2F_1$ hypergeometric functions.

\begin{thm}\label{thm:2F1congruence}
 Let $m$ and $d$ be integers with $1\leq m<d$. If $p\equiv 1\pmod d$ and $T$ is a generator for the character group $\widehat{\mathbb F_p^{\times}}$ then, for $x\not=0$,
 \begin{equation*}
   {}_{2}F_{1}\left(\left.\begin{array}{cc}
                \tfrac{m}{d}&\tfrac{d-m}{d}\\
		{} &1
               \end{array}\right|x\right)_{\text{tr}(p)} \equiv  -p\hspace{.05in}{}_{2}F_{1}\left(\left.\begin{array}{cc}
                T^{mt}&\overline T^{mt}\\
		{} &\epsilon
               \end{array}\right|x\right)_p \pmod p,
 \end{equation*}
where $t=\tfrac{p-1}{d}$.
\end{thm}

%\begin{rem*}
 %By testing values in Sage \cite{Sage} we see that one should not expect a congruence modulo higher powers of $p$ when $x\not=1$.
%\end{rem*}

The following corollary applies to the hypergeometric functions that appear in Clemens' and Koike's trace of Frobenius expressions for Legendre elliptic curves.
\begin{cor}\label{cor:2F1ECcongruence}
 If $p$ is an odd prime and $\phi$ is a quadratic character in $\widehat{\mathbb F_p^{\times}}$, then, for $\lambda\not=0$,
 \begin{equation*}
   (-1)^{\frac{p-1}{2}}{}_{2}F_{1}\left(\left.\begin{array}{cc}
                \tfrac12&\tfrac12\\
		{} &1
               \end{array}\right|\lambda\right)_{\text{tr}(p)} \equiv  -\phi(-1)p\hspace{.05in}{}_{2}F_{1}\left(\left.\begin{array}{cc}
               \phi&\phi\\
		{} &\epsilon
               \end{array}\right|\lambda\right)_p \pmod p.
 \end{equation*}
\end{cor}

\begin{proof}[Proof of Theorem \ref{thm:2F1congruence}]
 We use Equation \ref{eqn:HGFdef} to write
 \begin{align*}
  -p\hspace{.05in}{}_{2}F_{1}\left(\left.\begin{array}{cc}
               T^{mt}&\overline T^{mt}\\
		{} &\epsilon
               \end{array}\right|x\right)_p &=\frac{p^2}{1-p}\sum_{\chi}\binom{T^{mt}\chi}{\chi}\binom{\overline T^{mt}\chi}{\epsilon\chi}\chi(x)           
 \end{align*}
Using properties of Gauss and Jacobi sums from Sections \ref{sec:HGF} and \ref{sec:GaussJacobi} we break this down to
\begin{align*}
 \frac{p^2}{1-p}\sum_{\chi}\frac{\chi(-1)}{p}J(T^{mt}\chi,\overline\chi)&\frac{\chi(-1)}{p}J(\overline T^{mt}\chi,\overline\chi)\chi(x)\\
      &=\frac{1}{1-p}\sum_{\chi}\frac{g(T^{mt}\chi)g(\overline\chi)}{g(T^{mt})}\frac{g(\overline T^{mt}\chi)g(\overline\chi)}{g(\overline T^{mt})}\chi(x)\\
      &=\frac{1}{1-p}\sum_{\chi}\frac{g(T^{mt}\chi)g(\overline T^{mt}\chi)g(\overline\chi)^2}{T^{mt}(-1)p}\chi(x)\\
      &=\frac{\overline T^{mt}(-1)}{p(1-p)}\sum_{\chi}g(T^{mt}\chi)g(\overline T^{mt}\chi)g(\overline\chi)^2\chi(x).
\end{align*}
We rewrite this in terms of the Teichm\"uller character $\omega$ defined in Section \ref{sec:padicGamma} by letting $T=\overline\omega$ and $\chi=\overline\omega^{-j}$. Furthermore, we split up the sum as Mortenson does in \cite{Mortenson2003a}. This yields the expression
\begin{align*}
 \frac{\omega^{mt}(-1)}{p(1-p)}\sum_{\chi}g(T^{mt}\chi)g(\overline\chi)^2\chi(x)&=\frac{\overline T^{mt}(-1)}{p(1-p)}\left[\sum_{j=0}^{mt}g(\overline\omega^{mt-j})g(\overline\omega^{p-1-mt-j})g(\overline\omega^{j})^2 \omega^{j}(x)\right.\\
 &+\sum_{j=mt+1}^{p-1-mt}g(\overline\omega^{p-1+mt-j})g(\overline\omega^{p-1-mt-j})g(\overline\omega^{j})^2\omega^{j}(x)\\
 &+\left.\sum_{j=p-mt}^{p-2}g(\overline\omega^{p-1+mt-j})g(\overline\omega^{2(p-1)-mt-j})g(\overline\omega^{j})^2 \omega^{j}(x)\right].
\end{align*}
We use the Gross Koblitz formula given in Theorem \ref{thm:GrossKoblitz} to write this expression in terms of the $p$-adic Gamma function
\begin{align*}
 \frac{\omega^{mt}(-1)}{p(1-p)}&\left[\sum_{j=0}^{mt}\pi^{p-1}\Gamma_p\left(\tfrac{mt-j}{p-1}\right)\Gamma_p\left(\tfrac{p-1-mt-j}{p-1}\right)\Gamma_p\left(\tfrac{j}{p-1}\right)^2 \omega^{j}(x)\right.\\
 &+\sum_{j=mt+1}^{p-1-mt}\pi^{2(p-1)}\Gamma_p\left(\tfrac{p-1+mt-j}{p-1}\right)\Gamma_p\left(\tfrac{p-1-mt-j}{p-1}\right)\Gamma_p\left(\tfrac{j}{p-1}\right)^2 \omega^{j}(x)\\
 &+\left.\sum_{j=p-mt}^{p-2}\pi^{3(p-1)}\Gamma_p\left(\tfrac{p-1+mt-j}{p-1}\right)\Gamma_p\left(\tfrac{2(p-1)-mt-j}{p-1}\right)\Gamma_p\left(\tfrac{j}{p-1}\right)^2 \omega^{j}(x)\right].
\end{align*}
Recalling that $\pi$ is a solution to $x^{p-1}+p=0$, we have that $\pi^{p-1}=-p$. We use this to rewrite the sum as
\begin{align*}
  \frac{\omega^{mt}(-1)}{p(1-p)}&\left[\sum_{j=0}^{mt}-p\Gamma_p\left(\tfrac{mt-j}{p-1}\right)\Gamma_p\left(\tfrac{p-1-mt-j}{p-1}\right)\Gamma_p\left(\tfrac{j}{p-1}\right)^2 \omega^{j}(x)\right.\\
 &+\sum_{j=mt+1}^{p-1-mt}p^2\Gamma_p\left(\tfrac{p-1+mt-j}{p-1}\right)\Gamma_p\left(\tfrac{p-1-mt-j}{p-1}\right)\Gamma_p\left(\tfrac{j}{p-1}\right)^2 \omega^{j}(x)\\
 &+\left.\sum_{j=p-mt}^{p-2}-p^3\Gamma_p\left(\tfrac{p-1+mt-j}{p-1}\right)\Gamma_p\left(\tfrac{2(p-1)-mt-j}{p-1}\right)\Gamma_p\left(\tfrac{j}{p-1}\right)^2 \omega^{j}(x)\right].
\end{align*}
In Section \ref{sec:padicGamma} we define the Teichm\"uller character by the property $\omega(x)\equiv x \pmod p$ for all $x$ in $\{0,\ldots,p-1\}$. Using this we prove that the above sum is congruent modulo $p$ to the expression 
\begin{align*}
 \frac{(-1)^{mt}}{p(1-p)}&\left[\sum_{j=0}^{mt}-p\Gamma_p\left(\tfrac{mt-j}{p-1}\right)\Gamma_p\left(\tfrac{p-1-mt-j}{p-1}\right)\Gamma_p\left(\tfrac{j}{p-1}\right)^2 x^{j}\right.\\
 &+\sum_{j=mt+1}^{p-1-mt}p^2\Gamma_p\left(\tfrac{p-1+mt-j}{p-1}\right)\Gamma_p\left(\tfrac{p-1-mt-j}{p-1}\right)\Gamma_p\left(\tfrac{j}{p-1}\right)^2 x^{j}\\
 &+\left.\sum_{j=p-mt}^{p-2}-p^3\Gamma_p\left(\tfrac{p-1+mt-j}{p-1}\right)\Gamma_p\left(\tfrac{2(p-1)-mt-j}{p-1}\right)\Gamma_p\left(\tfrac{j}{p-1}\right)^2 x^{j}\right].
\end{align*}
We simplify this expression to get that this is congruent modulo $p$ to the expression 
\begin{align*}
\frac{(-1)^{mt}}{(1-p)}&\left[\sum_{j=0}^{mt}-\Gamma_p\left(\tfrac{mt-j}{p-1}\right)\Gamma_p\left(\tfrac{p-1-mt-j}{p-1}\right)\Gamma_p\left(\tfrac{j}{p-1}\right)^2 x^{j}\right.\\
 &+\sum_{j=mt+1}^{p-1-mt}p\Gamma_p\left(\tfrac{p-1+mt-j}{p-1}\right)\Gamma_p\left(\tfrac{p-1-mt-j}{p-1}\right)\Gamma_p\left(\tfrac{j}{p-1}\right)^2 x^{j}\\
 &+\left.\sum_{j=p-mt}^{p-2}-p^2\Gamma_p\left(\tfrac{p-1+mt-j}{p-1}\right)\Gamma_p\left(\tfrac{2(p-1)-mt-j}{p-1}\right)\Gamma_p\left(\tfrac{j}{p-1}\right)^2 x^{j}\right].
 \end{align*} 
 Note that since  $\Gamma_p: \mathbb Z_p \rightarrow \mathbb Z_p^*$, the last two summands are congruent to 0 modulo $p$. Hence, we are left with
 \begin{align*}
\frac{(-1)^{mt}}{(1-p)}\left[\sum_{j=0}^{mt}-\Gamma_p\left(\tfrac{mt-j}{p-1}\right)\Gamma_p\left(\tfrac{p-1-mt-j}{p-1}\right)\Gamma_p\left(\tfrac{j}{p-1}\right)^2 x^{j}\right] \pmod p.
\end{align*}
This much simpler expression can be broken down even further. Note that $\tfrac{mt}{p-1}=\tfrac{m}{d}$. Then the above expression is congruent, modulo $p$, to 
\begin{equation}\label{eqn:padicGammaSum}
-(-1)^{mt}\sum_{j=0}^{mt}\Gamma_p\left(\tfrac{m}{d} +j\right)\Gamma_p\left(1-\tfrac{m}{d} +j\right)\Gamma_p\left(-j\right)^2 x^{j}.
\end{equation}
We then use part 5 of Proposition \ref{prop:Gamma_p} to get
$$\Gamma_p(y)^2=\frac{1}{\Gamma_p(1-y)^2}.$$
We use this to write that Equation \ref{eqn:padicGammaSum} is congruent modulo $p$ to the expression 
\begin{align*}
(-1)^{mt+1}\sum_{j=0}^{mt}\frac{\Gamma_p\left(\tfrac{m}{d} +j\right)\Gamma_p\left(1-\tfrac{m}{d} +j\right)}{\Gamma_p\left(1+j\right)^2} x^{j}.
\end{align*}
We now use Proposition \ref{prop:MortensonPadic} to write that Equation \ref{eqn:padicGammaSum} is congruent modulo $p$ to the expression 
\begin{align*}
\sum_{j=0}^{mt}\frac{\left(\tfrac{m}{d}\right)_j\left(\tfrac{d-m}{d}\right)_j}{j!^2} x^{j}.
\end{align*}
Note that we are summing to $mt=\tfrac{m(p-1)}{d}$. If $j>mt$ then the rising factorial $\left(\tfrac{m}{d}\right)_j$ has as a factor
\begin{align*}
 \tfrac{m}{d}+\left(\tfrac{m(p-1)}{d}+ 1\right) -1 &= \tfrac{m}{d}+\tfrac{m(p-1)}{d}\\
						   %&=\tfrac{m}{d}(1+p-1)\\
						   &=p\cdot\tfrac{m}{d}.
\end{align*}
Thus, 
\begin{align*}
\sum_{j=0}^{mt}\frac{\left(\tfrac{m}{d}\right)_j\left(\tfrac{d-m}{d}\right)_j}{j!^2} x^{j}\equiv \sum_{j=0}^{p-1}\frac{\left(\tfrac{m}{d}\right)_j\left(\tfrac{d-m}{d}\right)_j}{j!^2} x^{j} \pmod p,
\end{align*}
since each missing term in the summand on the left has a factor of $p$ in it. \\

Noting that $(1)_j=j!$, we use Equation \ref{eqn:truncHGF} to identify this as a truncated hypergeometric series
\begin{align*}
   \sum_{j=0}^{p-1}\frac{\left(\tfrac{m}{d}\right)_j\left(\tfrac{d-m}{d}\right)_j}{j!^2} x^{j}&= {}_{2}F_{1}\left(\left.\begin{array}{cc}
                \tfrac{m}{d}&\tfrac{d-m}{d}\\
		{} &1
               \end{array}\right|x\right)_{\text{tr}(p)}.
\end{align*}
Thus, we have proved that
\begin{align*}
 {}_{2}F_{1}\left(\left.\begin{array}{cc}
                \tfrac{m}{d}&\tfrac{d-m}{d}\\
		{} &1
               \end{array}\right|x\right)_{\text{tr}(p)}\equiv  -p\hspace{.05in}{}_{2}F_{1}\left(\left.\begin{array}{cc}
               T^{mt}&\overline T^{mt}\\
		{} &\epsilon
               \end{array}\right|x\right)_p \pmod p.       
 \end{align*}
\end{proof}

Our next congruence result holds for hypergeometric functions that appear in our work with Dwork hypersurfaces. The proof is similar to that of Theorem \ref{thm:2F1congruence} and so we omit some steps.

\begin{thm}\label{thm:dFdcongruence}
 Let $p\equiv 1\pmod d$ and $T$ be a generator for the character group $\widehat{\mathbb F_p^{\times}}$ then,
 \begin{align*}
  p^{d-2}\hspace{.05in}{}_{d-1}F_{d-2}&\left(\left.\begin{array}{cccc}
                T^{t}&T^{2t}&\ldots&T^{(d-1)t}\\
		{} &\epsilon&\ldots&\epsilon
               \end{array}\right|x\right)_p \\
               &\equiv (-1)^d  {}_{d-1}F_{d-2}\left(\left.\begin{array}{cccc}
                \tfrac{1}{d}&\tfrac{2}{d}&\ldots&\tfrac{d-1}{d}\\
		{} &1&\ldots&1
               \end{array}\right|x\right)_{\text{tr}(p)}  \pmod p,
 \end{align*}
where $t=\tfrac{p-1}{d}$.
\end{thm}

We will use the following corollary in our work with Dwork K3 surfaces in Section \ref{sec:PeriodTrace}.
\begin{cor}\label{cor:3F2congruence}
 Let $p\equiv 1\pmod 4$ and $T$ be a generator for the character group $\widehat{\mathbb F_p^{\times}}$ then,
 \begin{equation}
  {}_{3}F_{2}\left(\left.\begin{array}{ccc}
                \tfrac{1}{4}&\tfrac{2}{4}&\tfrac{3}{4}\\
		{} &1&1
               \end{array}\right|x\right)_{\text{tr}(p)} \equiv  p^2\hspace{.05in}{}_{3}F_{2}\left(\left.\begin{array}{ccc}
                T^{t}&T^{2t}&T^{3t}\\
		{} &\epsilon&\epsilon
               \end{array}\right|x\right)_p \pmod p,
 \end{equation}
where $t=\tfrac{p-1}{4}$.
\end{cor}

\begin{proof}[Proof of Theorem \ref{thm:dFdcongruence}]
 We start by using Equation \ref{eqn:HGFdef} to write
 \begin{align*}
  p^{d-2}\hspace{.05in}{}_{d-1}F_{d-2}&\left(\left.\begin{array}{cccc}
                T^{t}&T^{2t}&\ldots&T^{(d-1)t}\\
		{} &\epsilon&\ldots&\epsilon
               \end{array}\right|x\right)_p&=\frac{p^{d-1}}{p-1}\sum_{\chi}\binom{T^{t}\chi}{\chi}\binom{T^{2t}\chi}\cdots{\epsilon\chi}\binom{T^{(d-1)t}\chi}{\epsilon\chi}\chi(x).           
 \end{align*}
Using properties of Gauss and Jacobi sums from Sections \ref{sec:HGF} and \ref{sec:GaussJacobi} we break this down to
\begin{align*}
 &\frac{1}{p-1}\sum_{\chi}J(T^{t}\chi,\overline\chi)J(T^{2t}\chi,\overline\chi)\cdots J(T^{(d-1)t}\chi,\overline\chi)\chi((-1)^{d-1}x)\\
 &\hspace{1in}=\frac{1}{p-1}\sum_{\chi}\frac{g(T^{t}\chi)g(T^{2t}\chi)\cdots g(T^{(d-1)t}\chi)g(\overline\chi)^{d-1}}{g(T^{t})g(T^{2t})\cdots g(T^{(d-1)t})}\chi((-1)^{d-1}x).
\end{align*}
We are able to simplify the denominator, but the result will be different depending on the parity of $d$. We split into two cases: $d$ even and $d$ odd.\\

We start by assuming $d$ is even. In this case our expression becomes
\begin{align*}
 =\frac{1}{p-1}\sum_{\chi}\frac{g(T^{t}\chi)g(T^{2t}\chi)\cdots g(T^{(d-1)t}\chi)g(\overline\chi)^{d-1}}{g(T^{t})g(T^{2t})\cdots g(T^{(d-1)t})}\chi(-x).
\end{align*}
Note that there are $d-1$ terms in the denominator, which is an odd number. We will have $\frac{d-2}{2}$ pairings of the form $g(T^{mt})g(T^{(d-m)t}) = pT^{mt}(-1)$. The remaining term that does not get paired off is $g(T^{td/2})$. Thus, our expression can be written as
\begin{align*}
= \frac{T^a(-1)}{p^{(d-2)/2}(p-1)g(T^{td/2})}\sum_{\chi}g(T^{t}\chi)g(T^{2t}\chi)\cdots g(T^{(d-1)t}\chi)g(\overline\chi)^{d-1}\chi(-x),
\end{align*}
where $a= t+2t+\ldots+\tfrac{d-2}{2}t$. We now rewrite this in terms of the Teichm\"uller character $\omega$ by letting $T=\overline\omega$ and $\chi=\overline\omega^{-j}$. Furthermore, we can split up the sum in a similar manner to how we did in the proof of Theorem \ref{thm:2F1congruence}. We will focus on the first summand only, with $0\leq j \leq t$, since, as in the previous proof, the remaining summands are congruent to 0 modulo $p$. Thus, we have
\begin{align*}
\frac{\overline\omega^a(-1)}{p^{(d-2)/2}(p-1)g(\overline\omega^{td/2})}\sum_{j=0}^{t}g(\overline\omega^{t-j})g(\overline\omega^{2t-j})\ldots g(\overline\omega^{(d-1)t-j})g(\overline\omega^{j})^{d-1} \omega^{j}(-x).
\end{align*}
We use the Gross Koblitz formula to write this expression in terms of the $p$-adic Gamma function. Note that the resulting power of $\pi$ cancels with the $p^{(d-2)/2}$ in the denominator.
\begin{align*}
=\frac{-\overline\omega^a(-1)}{(-1)^{(d-2)/2}(p-1)\Gamma_p\left(\tfrac{td/2}{p-1}\right)}\sum_{j=0}^{t}\Gamma_p\left(\tfrac{t-j}{p-1}\right)\Gamma_p\left(\tfrac{2t-j}{p-1}\right)\ldots\Gamma_p\left(\tfrac{(d-1)t-j}{p-1}\right)\Gamma_p\left(\tfrac{j}{p-1}\right)^{d-1} \omega^{j}(-x).
\end{align*}
We now reduce this sum modulo $p$ and simplify using the same techniques as in the proof of Theorem \ref{thm:2F1congruence}. We recall that $\omega(x)\equiv x \pmod p$ for all $x$ in $\{0,\ldots,p-1\}$ and $p-1\equiv -1 \pmod p$. Furthermore, by Equation \ref{eqn:padicdefinition} we have that $\Gamma_p(1+j)=(-1)^{1+j}j!$. We combine this with part 5 of Proposition \ref{prop:Gamma_p} to get
$$\Gamma_p(-j)^{d-1}=\frac{1}{j!^{d-1}}.$$
Lastly, if $0\leq j\leq t$, then $\Gamma_p\left(\frac{m}{d}+j\right)=(-1)^j\left(\tfrac{m}{d}\right)_j\Gamma_p\left(\frac{m}{d}\right)$, for $m=1,2,\ldots,d-1$. Thus, our expression is congruent modulo $p$ to 
\begin{align*}
\frac{(-1)^{a-(d-2)/2}}{\Gamma_p\left(\tfrac{d/2}{d}\right)}\left[\sum_{j=0}^{t}\frac{\left(\tfrac1d\right)_j\left(\tfrac2d\right)_j\cdots\left(\tfrac{d-1}{d}\right)_j\Gamma_p\left(\frac1d\right)\Gamma_p\left(\frac2d\right)\cdots\Gamma_p\left(\frac{d-1}{d}\right)}{j!^{d-1}}(x)^j\right].
\end{align*}
We use part 5 of Proposition \ref{prop:Gamma_p} to simplify the $p$-adic Gamma pairs
\begin{align*}
 \Gamma_p\left(\frac{m}{d}\right)\Gamma_p\left(1- \frac{m}{d}\right)&=(-1)^{(d-m)t+1}.
\end{align*}
The resulting exponent of $-1$ from this will be $b+\tfrac{d-2}{2}$, where $b=(d-1)t+\ldots + \left(d-\tfrac{d-2}{2}\right)t$. Thus, our expression is congruent modulo $p$ to 
\begin{align*}
 \equiv (-1)^{a+b}\sum_{j=0}^{t}\frac{\left(\tfrac1d\right)_j\left(\tfrac2d\right)_j\cdots\left(\tfrac{d-1}{d}\right)_j}{j!^{d-1}}(x)^j \pmod p,
\end{align*}
Note that 
\begin{align*}
 a+b&=t+2t+\ldots+\tfrac{d-2}{2}t+(d-1)t+\ldots + \left(d-\tfrac{d-2}{2}\right)t\\
    &=dt+\ldots+dt\\
    &=dt\cdot \tfrac{d-2}{2}\\
    &=(p-1)\cdot \tfrac{d-2}{2},
\end{align*}
which is even. Thus we can write that $(-1)^{a+b}=(-1)^{d}$. \\

Now suppose that $d$ is odd. In this case our expression becomes
\begin{align*}
 \frac{1}{p-1}\sum_{\chi}\frac{g(T^{t}\chi)g(T^{2t}\chi)\cdots g(T^{(d-1)t}\chi)g(\overline\chi)^{d-1}}{g(T^{t})g(T^{2t})\cdots g(T^{(d-1)t})}\chi(x).
\end{align*}

The rest of the proof is similar to the case when $d$ is even, but slightly easier. There are $d-1$ terms in the denominator, which is an even number. We will have $\frac{d-1}{2}$ pairings of the form $g(T^{mt})g(T^{(d-m)t}) = pT^{mt}(-1)$. Thus, our expression can be written as
\begin{align*}
\frac{T^{a'}(-1)}{p^{(d-1)/2}(p-1)}\sum_{\chi}g(T^{t}\chi)g(T^{2t}\chi)\cdots g(T^{(d-1)t}\chi)g(\overline\chi)^{d-1}\chi(x),
\end{align*}
where $a'= t+2t+\ldots+\tfrac{d-1}{2}t$.\\

We now rewrite this in terms of the Teichm\"uller character $\omega$ by letting $T=\overline\omega$ and $\chi=\overline\omega^{-j}$. Furthermore, we can split up the sum in a similar manner to how we did when $d$ was even. We will focus on the first summand only, with $0\leq j \leq t$, since, as in the previous proof, the remaining summands are congruent to 0 modulo $p$. Thus, we have
\begin{align*}
\frac{\overline\omega^{a'}(-1)}{p^{(d-1)/2}(p-1)}\sum_{j=0}^{t}g(\overline\omega^{t-j})g(\overline\omega^{2t-j})\ldots g(\overline\omega^{(d-1)t-j})g(\overline\omega^{j})^{d-1} \omega^{j}(x).
\end{align*}
We use the Gross Koblitz formula to write this expression in terms of the $p$-adic Gamma function
\begin{align*}
=\frac{\overline\omega^{a'}(-1)}{(-1)^{(d-1)/2}(p-1)}\sum_{j=0}^{t}\Gamma_p\left(\tfrac{t-j}{p-1}\right)\Gamma_p\left(\tfrac{2t-j}{p-1}\right)\ldots\Gamma_p\left(\tfrac{(d-1)t-j}{p-1}\right)\Gamma_p\left(\tfrac{j}{p-1}\right)^{d-1} \omega^{j}(x).
\end{align*}
Note that the resulting power of $\pi$ canceled with the $p^{(d-1)/2}$ in the denominator. Furthermore we reduce the sum modulo $p$ and simplify using the same techniques as in the proof for $d$ even. Our expression is congruent modulo $p$ to 
\begin{align*}
(-1)^{a'-(d-1)/2+1}\left[\sum_{j=0}^{t}\frac{\left(\tfrac1d\right)_j\left(\tfrac2d\right)_j\cdots\left(\tfrac{d-1}{d}\right)_j\Gamma_p\left(\frac1d\right)\Gamma_p\left(\frac2d\right)\cdots\Gamma_p\left(\frac{d-1}{d}\right)}{j!^{d-1}}(x)^j\right].
\end{align*}
We use part 5 of Proposition \ref{prop:Gamma_p} to simplify each term of the form
\begin{align*}
 \Gamma_p\left(\frac{m}{d}\right)\Gamma_p\left(1- \frac{m}{d}\right)&=(-1)^{(d-m)t+1}.
\end{align*}
The resulting exponent of $-1$ from this will be $b'+\tfrac{d-2}{2}$, where $b'=(d-1)t+\ldots + \left(d-\tfrac{d-2}{2}\right)t$. Thus, our expression is congruent modulo $p$ to 
\begin{align*}
(-1)^{a'+b'+1}\sum_{j=0}^{t}\frac{\left(\tfrac1d\right)_j\left(\tfrac2d\right)_j\cdots\left(\tfrac{d-1}{d}\right)_j}{j!^{d-1}}(x)^j.
\end{align*}
Note that $ a'+b'+1=\frac{d-1}{2}(p-1)+1$, which is an odd number. Thus, $(-1)^{a'+b'+1}=(-1)^{d}$, which matches the exponent of $-1$ in the case where $d$ was even.\\

We now bring the two cases together. For both even and odd $d$ we have
\begin{align*}
  p^{d-2}\hspace{.05in}{}_{d-1}F_{d-2}&\left(\left.\begin{array}{cccc}
                T^{t}&T^{2t}&\ldots&T^{(d-1)t}\\
		{} &\epsilon&\ldots&\epsilon
               \end{array}\right|x\right)_p&\equiv(-1)^d \sum_{j=0}^{t}\frac{\left(\tfrac1d\right)_j\left(\tfrac2d\right)_j\cdots\left(\tfrac{d-1}{d}\right)_j}{j!^{d-1}}(x)^j \pmod p .
 \end{align*}
We use Equation \ref{eqn:truncHGF} to identify this as a truncated hypergeometric series 
\begin{align*}
 \equiv (-1)^d  {}_{d-1}F_{d-2}\left(\left.\begin{array}{cccc}
                \tfrac{1}{d}&\tfrac{2}{d}&\ldots&\tfrac{d-1}{d}\\
		{} &1&\ldots&1
               \end{array}\right|x\right)_{\text{tr}(p)} \pmod p ,
\end{align*}
where, as in the proof of Theorem \ref{thm:2F1congruence}, the terms with $j>t$ are congruent to 0 modulo $p$.\\

\end{proof}

\section{Koblitz's Point Count Formula}
\label{sec:Koblitz}

Koblitz \cite{KoblitzHypersurface} developed a formula for the number of points on diagonal hypersurfaces in the Dwork family in terms of Gauss sums. We specialize this formula to the family of Dwork K3 surfaces, i.e. to the case when $d,n=4, h=1$. \\

Let $W$ be the set of all 4-tuples $w=(w_1,w_2,w_3,w_4)$ satisfying $0\leq w_i<4$ and $\sum w_i\equiv 0 \pmod 4$. We denote the points on the diagonal hypersurface
$$x_1^4+x_2^4+x_3^4+x_4^4=0$$
by $N_q(0):=\sum N_{q}(0,w)$, where
$$
 N_{q}(0,w)=
 \begin{cases}
  0 &\text{if some but not all } w_i=0,\\
  \frac{q^3-1}{q-1} &\text{if all } w_i=0,\\
  -\frac1q J\left(T^{\tfrac{w_1}{4}},T^{\tfrac{w_2}{4}},T^{\tfrac{w_3}{4}},T^{\tfrac{w_4}{4}}\right) &\text{if all } w_i\not=0.\\
 \end{cases}
$$

\begin{thm}\cite[Theorem 2]{KoblitzHypersurface}\label{thm:KoblitzHypersurface}
The number of points on the Dwork K3 surface $X_{\lambda}^4$ is given by
$$\#X_{\lambda}^4(\mathbb F_q)=N_q(0)+\frac1{q-1}\sum\frac{\prod_{i=1}^4g\left(T^{w_it+j}\right)}{g(T^{4j})}T^{4j}(4\lambda)$$
where the sum is taken over $j\in\{0,\ldots,q-2\}$ and $w=(w_1,w_2,w_3,w_4)$ in $W$.\\

\end{thm}

We wish to simplify this formula. We start by considering the term $N_q(0)$.

\begin{prop}\label{prop:N0}
 $N_q(0)=q^2+7q+1  + \frac1q\sum_{i=1}^3g(T^{it})^4 +12qT^t(-1)$.
\end{prop}
\begin{proof}

The list of possible 4-tuples (up to reordering since the order doesn't matter in the Jacobi sum) is
$$W^*=\{ (1,1,1,1)^1, (2,2,2,2)^1, (3,3,3,3)^1, (1,1,3,3)^6, (1,2,2,3)^{12}\},$$
$w=(0,0,0,0)$, and all of the 4-tuples where some, not all, of the $w_i=0$ (we exclude this list since $N_{q}(0,w)=0$ for these tuples). Letting $t=\frac{q-1}{4}$ we see that
$$J\left(T^{\tfrac{w_1}{4}},T^{\tfrac{w_2}{4}},T^{\tfrac{w_3}{4}},T^{\tfrac{w_4}{4}}\right)= -\prod_i g(T^{w_it}).$$
Thus, we have
\begin{align*}
 N_q(0) &= \sum_w N_q(0,w)\\
 &=\frac{q^3-1}{q-1} +\sum_{w\in W^*}N_q(0,w)\\
 &=q^2+q+1 + \frac1q\left[\sum_{i=1}^3g(T^{it})^4 +6g(T^t)^2g(T^{3t})^2+12g(T^t)g(T^{2t})^2g(T^{3t}) \right]\\
 %&=q^2+q+1  + \frac1q\left[\sum_{i=1}^3g(T^{it})^4 +6q^2T^t(-1)T^{t}(-1)+12q^2T^t(-1)T^{2t}(-1) \right]\\
 &=q^2+q+1  + \frac1q\left[\sum_{i=1}^3g(T^{it})^4 +6q^2+12q^2T^t(-1) \right]\\
 &=q^2+7q+1  + \frac1q\sum_{i=1}^3g(T^{it})^4 +12qT^t(-1).
\end{align*}
\end{proof}

Define the equivalence relation $\sim$ on $W$ by $w\sim w'$ if $w-w'$ is a multiple of $(1,1,1,1)$. Up to permutation, there are three cosets (up to permutation) in $W/\sim$:
$$(0,0,0,0)^1, (0,1,1,2)^{12}, (0,0,2,2)^3.$$
The first coset contains the obvious four elements. The second set of 12 cosets is made up of permutations of the coset $\overline{(0,1,1,2)}=\{(0,1,1,2),(1,2,2,3), (2,3,3,0), (3,0,0,1)\}$. The third set of 3 cosets is made up of permutations of the coset $\overline{(0,0,2,2)}=\{(0,0,2,2),(1,1,3,3)\}$. This last set of cosets is different in that some permutations are in the same coset. For example, the element $(2,2,0,0)$ is a permutation of $(0,0,2,2)$ but they are also in the same coset.\\

Let \begin{equation}\label{eqn:DefS_w}
 S_{[w]}=\frac{1}{q-1}\sum_{j=0}^{q-2}\frac{\prod_{i=1}^4g\left(T^{w_it +j}\right)}{g\left(T^{4j}\right)}T^{4j}(4\lambda)    
    \end{equation}
denote the summands corresponding to all $w'\in[w]$. Our main tool for simplifying terms of this form is the Hasse-Davenport relation for Gauss sums.\\

In the propositions that follow, we give explicit formulas for each $S_{[w]}$.

\begin{prop}\label{prop:0000}
Let $w=(0,0,0,0)$. Then 
$$S_{[w]}=-\frac{1}{q}\sum_{i=1}^3g\left(T^{it}\right)^4+q^2{}_{3}F_{2}\left(\left.\begin{array}{ccc}
                T^t&T^{2t}&T^{3t}\\
		{} &\epsilon&\epsilon
               \end{array}\right|\frac{1}{\lambda^4}\right)_q   $$
\end{prop}
\begin{rem*}
 Note that the term $-\frac{1}{q}\sum_{i=1}^3g\left(T^{it}\right)^4$ negates a term in the overall point count (see Theorem \ref{prop:N0}).\\
\end{rem*}

\begin{proof}
By Equation \ref{eqn:DefS_w} we have
\begin{align*}
 S_{(0,0,0,0)}&=\frac{1}{q-1}\sum_{j=0}^{q-2}\frac{g\left(T^{j}\right)^4}{g\left(T^{4j}\right)}T^{4j}(4\lambda).
\end{align*}
If $t\mid j$, i.e. if $j=t,2t,3t$, then 
\begin{align*}
 \frac{g\left(T^{j}\right)^4}{g\left(T^{4j}\right)}T^{4j}(4\lambda)&=-g\left(T^{j}\right)^4.
\end{align*}
Thus,
\begin{align*}
 S_{(0,0,0,0)}&=-\frac1{q-1}\sum_{i=1}^3g\left(T^{it}\right)^4 + \frac{1}{q-1}\sum_{j=0, t\nmid j}^{q-2}\frac{g\left(T^{j}\right)^4}{g\left(T^{4j}\right)}T^{4j}(4\lambda)\\
 	      &= -\frac1{q-1}\sum_{i=1}^3g\left(T^{it}\right)^4 + \frac{1}{q-1}\sum_{j=0, t\nmid j}^{q-2}\frac{g\left(T^{j}\right)^4g\left(T^{-4j}\right)}{T^{4j}(-1)q}T^{4j}(4\lambda)\\
	      &= -\frac1{q-1}\sum_{i=1}^3g\left(T^{it}\right)^4 + \frac{1}{q(q-1)}\sum_{j=0, t\nmid j}^{q-2}g\left(T^{j}\right)^4g\left(T^{-4j}\right)T^{4j}(4\lambda).
\end{align*}
Note that if $t\mid j$ then
\begin{align*}
 g\left(T^{j}\right)^4g\left(T^{-4j}\right)T^{4j}(4\lambda)&=-g\left(T^{j}\right)^4.
\end{align*}
Hence,
\begin{align*}
 S_{(0,0,0,0)}&=-\frac1{q-1}\sum_{i=1}^3g(T^{it})^4 +\frac{1}{q(q-1)}\sum_{i=1}^3g(T^{it})^4+ \frac{1}{q(q-1)}\sum_{j=0}^{q-2}g(T^{j})^4g(T^{-4j})T^{4j}(4\lambda)\\
 &=\frac{-q+1}{q(q-1)}\sum_{i=1}^3g(T^{it})^4 +\frac{1}{q(q-1)}\sum_{j=0}^{q-2}g(T^{j})^4g(T^{-4j})T^{4j}(4\lambda)\\
 &=-\frac{1}{q}\sum_{i=1}^3g(T^{it})^4 +\frac{1}{q(q-1)}\sum_{j=0}^{q-2}g(T^{j})^4g(T^{-4j})T^{4j}(4\lambda).
\end{align*}
Working from the other direction we see that
 \begin{align*}
  {}_{3}F_{2}\left(\left.\begin{array}{ccc}
                T^t&T^{2t}&T^{3t}\\
		{} &\epsilon&\epsilon
               \end{array}\right|\frac{1}{\lambda^4}\right)_q &= \frac{q}{q-1}\sum_{\chi}\binom{T^t\chi}{\chi}\binom{T^{2t}\chi}{\epsilon\chi}\binom{T^{3t}\chi}{\epsilon\chi}\chi(1/{\lambda^4})\\
               &=\frac{q}{q-1}\sum_{\chi} \left(\frac{\chi(-1)}{q}\right)^3 J(T^{t}\chi,\overline\chi)J(T^{2t}\chi,\overline\chi)J(T^{3t}\chi,\overline\chi)\overline\chi({\lambda^4})\\
               &=\frac{1}{q^2(q-1)}\sum_{\chi} \frac{g(T^{t}\chi)g(T^{2t}\chi)g(T^{3t}\chi)g(\overline\chi)^3}{\prod_{i=1}^3 g(T^{it})}\overline\chi(-{\lambda^4}).
 \end{align*}
Use the Hasse-Davenport relation
$$\prod_{i=1}^3\frac{g\left(T^{it}\chi\right)}{g\left(T^{it}\right)}=\frac{g(\chi^4)\chi^{-4}(4)}{g(\chi)}$$
to get 
\begin{align*}
{}_{3}F_{2}\left(\left.\begin{array}{ccc}
                T^t&T^{2t}&T^{3t}\\
		{} &\epsilon&\epsilon
               \end{array}\right|\frac{1}{\lambda^4}\right)_q &= \frac{1}{q^2(q-1)}\sum_{\chi} \frac{g(\chi^4)\chi^{-4}(4)g(\overline\chi)^3}{g(\chi)}\overline\chi(-{\lambda^4})\\
                &= \frac{1}{q^2(q-1)}\sum_{\chi} \frac{g(\chi^4)\chi^{-4}(4)g(\overline\chi)^4}{\chi(-1)q}\overline\chi(-{\lambda^4})\\
                %&= \frac{1}{q^3(q-1)}\sum_{\chi} g(\chi^4)\chi^{-4}(4)g(\overline\chi)^4\overline\chi({\lambda^4})\\
                &= \frac{1}{q^3(q-1)}\sum_{j=0}^{q-2}g(T^j)^4 g(T^{-4j})T^{4j}(4{\lambda}),                
\end{align*}
which proves the desired result.
\end{proof}

\begin{prop}\label{prop:0112}
 Let $w=(0,1,1,2)$. Then 
$$S_{[w]}=12qT^t(-1)\left(T^{2t}(1-\lambda^4)-1\right).$$
\end{prop}

\begin{rem*}
 Note that the term $-12qT^t(-1)$ negates a term in the overall point count (see Theorem \ref{prop:N0}) Also note that $T^{2t}(1-\lambda^4)=0$ when $\lambda^4=1$, so the final expression for the point count when $\lambda^4=1$ is as simple as we might expect. 
\end{rem*}

\begin{cor}\label{cor:0112b}
 Let $w=(0,1,1,2)$ and $\lambda^4=1$. Then 
$$S_{[w]}=-12qT^t(-1).$$
\end{cor}

\begin{proof}[Proof of Proposition \ref{prop:0112}]
By Equation \ref{eqn:DefS_w} we have
 \begin{align*}
  S_{(0,1,1,2)}&=\frac{12}{q-1}\sum_{j=0}^{q-2}\frac{g\left(T^{j}\right)g\left(T^{t+j}\right)^2g\left(T^{2t+j}\right)}{g\left(T^{4j}\right)}T^{4j}(4\lambda).
 \end{align*}
  If $j=t$ then 
  \begin{align*}
   \frac{g\left(T^{j}\right)g\left(T^{t+j}\right)^2g\left(T^{2t+j}\right)}{g\left(T^{4j}\right)}T^{4j}(4\lambda)&=\frac{g\left(T^{t}\right)g\left(T^{2t}\right)^2g\left(T^{3t}\right)}{g\left(T^{4t}\right)}T^{4t}(4\lambda)\\
   &=\frac{T^t(-1)q\cdot T^{2t}(-1)q}{-1}\\
   &=-T^t(-1)q^2.
  \end{align*}
 For the terms with $j\not=t$ we use Corollary \ref{cor:HasseDavenport} to write
 \begin{align*}
 \sum_{j=0}^{q-2}\frac{g\left(T^{j}\right)g\left(T^{t+j}\right)^2g\left(T^{2t+j}\right)}{g\left(T^{4j}\right)}T^{4j}(4\lambda)
 &=g(T^{2t})T^t(-1)q\\
               &\hspace{.2in}\times\sum_{j=0, j\not=t}^{q-2}\frac{g\left(T^{j}\right)g\left(T^{t+j}\right)^2g\left(T^{2t+j}\right)}{g\left(T^{j}\right)g\left(T^{t+j}\right)g\left(T^{2t+j}\right)g\left(T^{3t+j}\right)}T^{4j}(\lambda)\\
  &=g(T^{2t})T^t(-1)q\sum_{j=0, j\not=t}^{q-2}\frac{g\left(T^{t+j}\right)}{g\left(T^{3t+j}\right)}T^{4j}(\lambda)\\
  &=g(T^{2t})T^t(-1)q\sum_{j=0, j\not=t}^{q-2}\frac{g\left(T^{t+j}\right)g\left(T^{t-j}\right)}{T^{3t+j}(-1)q}T^{4j}(\lambda)\\
  &=g(T^{2t})\sum_{j=0, j\not=t}^{q-2}g\left(T^{t+j}\right)g\left(T^{t-j}\right)T^j(-1)T^{4j}(\lambda).
 \end{align*}
 Note that if $j=t$ then 
 \begin{align*}
  g(T^{2t})\cdot g\left(T^{t+j}\right)g\left(T^{t-j}\right)T^j(-1)T^{4j}(\lambda)&=g\left(T^{2t}\right)^2g\left(T^{0}\right)T^t(-1)T^{4t}(\lambda)\\
  &=-T^t(-1)q.
 \end{align*}
Thus, 
\begin{align*}
 &g(T^{2t})\sum_{j=0, j\not=t}^{q-2}g(T^{t+j})g(T^{t-j})T^j(-1)T^{4j}(\lambda)\\
 &\hspace{1in}=g(T^{2t})\sum_{j=0}^{q-2}g(T^{t+j})g(T^{t-j})T^j(-1)T^{4j}(\lambda)+T^t(-1)q\\
 &\hspace{1in}=q(q-1)T^t(-1)T^{2t}(1-\lambda^4)+T^t(-1)q
\end{align*}
by Proposition \ref{prop:gaussproduct}. Hence,
\begin{align*}
   S_{(0,1,1,2)}&=\frac{12}{q-1}\left(q(q-1)T^t(-1)T^{2t}(1-\lambda^4)+T^t(-1)q -T^t(-1)q^2\right)\\
   &=\frac{12}{q-1}\left(q(q-1)T^t(-1)T^{2t}(1-\lambda^4)-T^t(-1)q(q-1)\right)\\
   &=12qT^t(-1)\left(T^{2t}(1-\lambda^4)-1\right).
\end{align*}

\end{proof}

The remaining piece of the point count formula for the Dwork K3 surface family is the term associated to $S_{[w]}$ with $w=(0,0,2,2)$. It's interesting that this term can be expressed in terms of a ${}_2F_1$ hypergeometric function. Note that the ${}_2F_1$ that appears is of a different form than what has been observed in other point-count formulas \cite{AhlgrenOno00a,Fuselier10,Koike1995, Lennon1, McCarthy2012b, Mortenson2003a, Swisher2015Arxiv}, where the lower characters are all the trivial character.

\begin{prop}\label{prop:0022}
 Let $w=(0,0,2,2)$. 
For all $\lambda\not\equiv 0 \pmod q$,
$$S_{[w]}=-6q+3q^2\binom{T^{3t}}{T^t}{}_{2}F_{1}\left(\left.\begin{array}{cc}
                T^{3t}&T^{t}\\
		{} &T^{2t}
               \end{array}\right|\frac{1}{\lambda^4}\right)_q.$$
\end{prop}
\begin{cor}\label{cor:0022}
 As we will see in Equation \ref{eqn:S0022lambda1}, this expression simplifies nicely when $\lambda^4=1$. In this case, 
$$S_{[w]}=-6q+3qT^t(-1).$$
\end{cor}

\begin{proof}[Proof of Proposition \ref{prop:0022}]
 The proof starts out similar to the proof of Proposition \ref{prop:0112}. By definition we have
 \begin{equation}\label{eqn:S0022}
  S_{(0,0,2,2)}=\frac{3}{q-1}\sum_{j=0}^{q-2}\frac{g\left(T^{j}\right)^2g\left(T^{2t+j}\right)^2}{g\left(T^{4j}\right)}T^{4j}(4\lambda).
 \end{equation}
 We use Corollary \ref{cor:HasseDavenport} to write
 \begin{align*}
  S_{(0,0,2,2)}&=\frac{3g(T^{2t})T^t(-1)q}{q-1}\sum_{j=0}^{q-2}\frac{g\left(T^{j}\right)^2g\left(T^{2t+j}\right)^2}{g\left(T^{j}\right)g\left(T^{t+j}\right)g\left(T^{2t+j}\right)g\left(T^{3t+j}\right)}T^{4j}(\lambda)\\
               &=\frac{3g(T^{2t})T^t(-1)q}{q-1}\sum_{j=0}^{q-2}\frac{g\left(T^{j}\right)g\left(T^{2t+j}\right)}{g\left(T^{t+j}\right)g\left(T^{3t+j}\right)}T^{4j}(\lambda).
 \end{align*}
 If $j=t$ then 
  \begin{align*}
   \frac{g\left(T^{j}\right)g\left(T^{2t+j}\right)}{g\left(T^{t+j}\right)g\left(T^{3t+j}\right)}T^{4j}(\lambda)&=\frac{g\left(T^{t}\right)g\left(T^{3t}\right)}{g\left(T^{2t}\right)g\left(T^{4t}\right)}T^{4t}(\lambda)\\
   &=\frac{T^t(-1)q}{-g(T^{2t})}.
  \end{align*}
  Similarly, if $j=3t$ then 
  \begin{align*}
   \frac{g\left(T^{j}\right)g\left(T^{2t+j}\right)}{g\left(T^{t+j}\right)g\left(T^{3t+j}\right)}T^{4j}(\lambda)%&=\frac{g\left(T^{3t}\right)g\left(T^{t}\right)}{g\left(T^{4t}\right)g\left(T^{2t}\right)}T^{12t}(\lambda)\\
   &=\frac{T^t(-1)q}{-g(T^{2t})}.
  \end{align*}
For the terms with $j\not=t,3t$ we have
\begin{align*}
 \sum_{j=0,j\not=t,3t}^{q-2}\frac{g\left(T^{j}\right)g\left(T^{2t+j}\right)}{g\left(T^{t+j}\right)g\left(T^{3t+j}\right)}T^{4j}(\lambda)&=\sum_{j=0,j\not=t,3t}^{q-2}\frac{g\left(T^{j}\right)g\left(T^{2t+j}\right)g\left(T^{t-j}\right)g\left(T^{3t-j}\right)}{T^{t+j}(-1)q\cdot T^{3t+j}(-1)q}T^{4j}(\lambda)\\
 &=\frac1{q^2}\sum_{j=0,j\not=t,3t}^{q-2}g\left(T^{j}\right)g\left(T^{2t+j}\right)g\left(T^{t-j}\right)g\left(T^{3t-j}\right)T^{4j}(\lambda).
\end{align*}
Note that if $j=t$ then 
\begin{align*}
 \frac{1}{q^2}g\left(T^{j}\right)g\left(T^{2t+j}\right)g\left(T^{t-j}\right)g\left(T^{3t-j}\right)T^{4j}(\lambda)&=\frac{1}{q^2}g\left(T^{t}\right)g\left(T^{3t}\right)g\left(T^{0}\right)g\left(T^{2t}\right)T^{4t}(\lambda)\\
 %&=-\frac{g(T^{2t})T^t(-1)q}{q^2}\\
 &=-\frac{g(T^{2t})T^t(-1)}{q}.
\end{align*}
The same holds for $j=3t$. Hence,
\begin{align*}
 S_{(0,0,2,2)}&=\frac{3g(T^{2t})T^t(-1)q}{q-1}\left[\frac1{q^2}\sum_{j=0}^{q-2}g(T^{j})g(T^{2t+j})g(T^{t-j})g(T^{3t-j})T^{4j}(\lambda)\right.\\
      &\hspace{2in} \left.+\frac{2g(T^{2t})T^t(-1)}{q}+\frac{2T^t(-1)q}{-g(T^{2t})}\right].
\end{align*}
Observe that
\begin{align*}
 \frac{3g(T^{2t})T^t(-1)q}{q-1}\left[\frac{2g(T^{2t})T^t(-1)}{q}+\frac{2T^t(-1)q}{-g(T^{2t})}\right]&=\frac{6g(T^{2t})}{q-1}\left[g(T^{2t})-\frac{q^2}{g(T^{2t})}\right]\\
      &=\frac{6}{q-1}\left(T^{2t}(-1)q-{q^2}\right)\\
      %&=\frac{6}{q-1}\left(q(1-q)\right)\\
      &=-6q.
\end{align*}
To simplify
\begin{equation}\label{eqn:S0022a}
\frac{3g(T^{2t})T^t(-1)q}{q-1}\left[\frac1{q^2}\sum_{j=0}^{q-2}g(T^{j})g(T^{2t+j})g(T^{t-j})g(T^{3t-j})T^{4j}(\lambda)\right], 
\end{equation}
we consider two cases. We first restrict to the case where $\lambda^4=1$ and apply Theorem \ref{thm:Helversen1978} to our formula to get
\begin{align*}
 \frac1{q-1}\sum_{j=0}^{q-2}g(T^{j})g(T^{2t+j})g(T^{t-j})g(T^{3t-j})&=\frac{g(T^t)g(T^{3t})g(T^{3t})g(T^t)}{g(T^{2t})}\\
      %&=\frac{T^t(-1)q\cdot T^t(-1)q}{g(T^{2t})}\\
      &=\frac{q^2}{g(T^{2t})}.
\end{align*}
Hence, when $\lambda^4=1$ we have
\begin{align}\label{eqn:S0022lambda1}
\begin{split}
 S_{(0,0,2,2)}&=-6q+\frac{3g(T^{2t})T^t(-1)q}{q^2}\left(\frac{q^2}{g(T^{2t})}\right)\\
      &=-6q+3qT^t(-1).
      \end{split}
\end{align}
For all other $\lambda\not=0$ we proceed as we did in the proof of Proposition \ref{prop:gaussproduct}. Recalling that $g(\chi)=\sum_x \chi(x)\theta(x)$ we can write
\begin{align*}
 &\sum_{j=0}^{q-2}g(T^{j})g(T^{2t+j})g(T^{t-j})g(T^{3t-j})T^{4j}(\lambda)\\
 &\hspace{1in}= \sum_{x,y,z,w}T^{2t}(y)T^t(z)T^{3t}(w)\theta(x+y+z+w)\sum_{j=0}^{q-2}T^j\left(\tfrac{xy\lambda^4}{zw}\right),
\end{align*}
where $x,y,z,w\not=0$. Note that $T^j\left(\tfrac{xy\lambda^4}{zw}\right)=q-1$ if $\tfrac{xy\lambda^4}{zw}=1$ and equals 0 otherwise. Hence, letting $x=\tfrac{zw}{y\lambda^4}$, the sum simplifies to
\begin{align*}
 (q-1)\sum_{y,z,w}T^{2t}(y)T^t(z)T^{3t}(w)\theta\left(\tfrac{zw}{y\lambda^4}+y+z+w\right).
\end{align*}
Since $y$ and $\lambda$ are both nonzero, we can perform the change of variables $z\rightarrow zy\lambda^4$ and get
\begin{align*}
 (q-1)\sum_{y,z,w}T^{2t}(y)T^t(zy\lambda^4)T^{3t}(w)\theta\left(zw+y+zy\lambda^4+w\right).
\end{align*}
Since $T^t(\lambda^4)=1$ we get
 \begin{align*}
 (q-1)\sum_{y,z,w}T^{3t}(y)T^t(z)T^{3t}(w)\theta\left(w(z+1)+y(1+z\lambda^4)\right).
\end{align*}
Note that if $z=-1$ then the above expression equals
\begin{align*}
 (q-1)\sum_{y,w}T^{3t}(y)T^t(z)T^{3t}(w)\theta\left(y(1-\lambda^4)\right)&=(q-1)\sum_{y}T^{3t}(y)T^t(z)\theta\left(y(1-\lambda^4)\right)\sum_wT^{3t}(w),
\end{align*}
which equals 0 since $T^{3t}\not=\epsilon$ implies $\sum_wT^{3t}(w)=0$. Similarly, the expression equals 0 when $z=-\lambda^{-4}$.\\

For $z\not=-1, -\lambda^{-4}$ we can perform the changes of variables $w\rightarrow\tfrac{w}{z+1}$ and $y\rightarrow\tfrac{y}{1+z\lambda^{4}}$. This portion of the sum then becomes
\begin{align*}
 (q-1)\sum_{y,w}T^{3t}(y)T^{3t}(w)\theta(w+y)\sum_{z\not=-1,-\lambda^{-4}}T^t(z)T^{-3t}(1+z\lambda^{4})T^{-3t}(z+1),
\end{align*}
where $$\sum_{z\not=-1,-\lambda^{-4}}T^t(z)T^{-3t}(1+z\lambda^{4})T^{-3t}(z+1)=\sum_{z\not=-1,-\lambda^{-4}}T^t(z)T^{t}(1+z\lambda^{4})T^{t}(z+1).$$ Note that if $z=-1$ or $z=-\lambda^{-4}$ then
$$T^t(z)T^{t}(1+z\lambda^{4})T^{t}(z+1)=0$$
so that we can include these $z-$values in the sum to get
\begin{align*}
 (q-1)\sum_{y,w}T^{3t}(y)T^{3t}(w)\theta(w+y)\sum_{z}T^t(z)T^{t}(1+z\lambda^{4})T^{t}(z+1).
\end{align*}
 This expression reduces further since
$$\sum_{y,w}T^{3t}(y)T^{3t}(w)\theta(w+y)=g(T^{3t})g(T^{3t}).$$
Furthermore, using the change of variables $z\rightarrow -z$ we see that
\begin{align*}
 \sum_{z}T^t(z)T^{t}(1+z\lambda^{4})T^{t}(z+1)&=\sum_{z}T^t(-z)T^{t}(1-z\lambda^{4})T^{t}(-z+1)\\
	&=T^t(-1)\sum_{z}T^t(z)T^{t}(1-z\lambda^{4})T^{t}(1-z)\\
	&=\frac{q}{\epsilon(\lambda^4)}{}_{2}F_{1}\left(\left.\begin{array}{cc}
                T^{3t}&T^{t}\\
		{} &T^{2t}
               \end{array}\right|\lambda^4\right)_q
\end{align*}
where the last expression is obtained by using Equation \ref{eqn:2F1def} with $A=T^{3t}, B=T^t$, and $C=T^{2t}$.\\

Thus we have shown that the original Gauss sum expression can be written as
$$q(q-1)g(T^{3t})^2\frac{q}{\epsilon(\lambda^4)}{}_{2}F_{1}\left(\left.\begin{array}{cc}
                T^{3t}&T^{t}\\
		{} &T^{2t}
               \end{array}\right|\lambda^4\right)_q.$$
               
Putting all of this work together leads to
\begin{align*}
  S_{(0,0,2,2)}&=-6q+\frac{3g(T^{2t})T^t(-1)q}{q-1}\cdot\frac1{q^2}\left[q(q-1)g(T^{3t})^2\frac{q}{\epsilon(\lambda^4)}{}_{2}F_{1}\left(\left.\begin{array}{cc}
                T^{3t}&T^{t}\\
		{} &T^{2t}
               \end{array}\right|\lambda^4\right)_q\right]\\
               &=-6q+3g(T^{2t})g(T^{3t})^2T^t(-1){}_{2}F_{1}\left(\left.\begin{array}{cc}
                T^{3t}&T^{t}\\
		{} &T^{2t}
               \end{array}\right|\lambda^4\right)_q,
\end{align*}
for  $\lambda^4\not=1$.\\

We now show that this expression may also be used in the case where $\lambda^4=1$. We use Equation \ref{eqn:2F1def} and properties of Gauss and Jacobi sums to get
\begin{align*}
 &3g(T^{2t})g(T^{3t})^2T^t(-1){}_{2}F_{1}\left(\left.\begin{array}{cc}
                T^{3t}&T^{t}\\
		{} &T^{2t}
               \end{array}\right|1\right)_q\\
               &\hspace{1in}= 3g(T^{2t})g(T^{3t})^2T^t(-1)\cdot\frac{T^{3t}(-1)}{q}\sum_y T^t(y)T^t(1-y)T^t(1-y)\\
               &\hspace{1in}=\frac{3g(T^{2t})g(T^{3t})^2}{q}\sum_y T^t(y)T^{2t}(1-y)\\
               &\hspace{1in}=\frac{3g(T^{2t})g(T^{3t})^2}{q}J(T^t,T^{2t})\\
               &\hspace{1in}=\frac{3g(T^{2t})g(T^{3t})^2}{q}\frac{g(T^t)g(T^{2t})}{g(T^{3t})}\\
               %&\hspace{1in}=\frac{3g(T^{2t})g(T^{3t})g(T^t)g(T^{2t})}{q}\\
              % &\hspace{1in}=\frac{3q^2T^t(-1)}{q}\\
               &\hspace{1in}=3qT^t(-1).
\end{align*}
Hence, for all $\lambda\not\equiv0 \pmod q$, 
\begin{equation*}
 S_{[w]}=-6q+3g(T^{2t})g(T^{3t})^2T^t(-1){}_{2}F_{1}\left(\left.\begin{array}{cc}
                T^{3t}&T^{t}\\
		{} &T^{2t}
               \end{array}\right|\lambda^4\right)_q.
\end{equation*}
Note that $g(T^{3t})^2=g(T^{2t})J(T^{3t},T^{3t})$, hence
\begin{align*}
 g(T^{2t})g(T^{3t})^2T^t(-1)&=g(T^{2t})^2J(T^{3t},T^{3t})T^t(-1)\\
			    &=qT^t(-1)J(T^{3t},T^{3t})\\
			    &=q^2\binom{T^{3t}}{T^t},
\end{align*}
where the last equality holds by Equation \ref{eqn:normalizedjacobi}. Thus we can write
$$S_{[w]}=-6q+3q^2\binom{T^{3t}}{T^t}{}_{2}F_{1}\left(\left.\begin{array}{cc}
                T^{3t}&T^{t}\\
		{} &T^{2t}
               \end{array}\right|\lambda^4\right)_q.$$
               
We finish the proof by using Theorem 4.2 of \cite{Greene} to rewrite the hypergeometric term.
\begin{align*}
{}_{2}F_{1}\left(\left.\begin{array}{cc}
                T^{3t}&T^{t}\\
		{} &T^{2t}
               \end{array}\right|\lambda^4\right)_q &= T^{2t}(-1)T^{3t}(\lambda^4){}_{2}F_{1}\left(\left.\begin{array}{cc}
                T^{3t}&T^{t}\\
		{} &T^{2t}
               \end{array}\right|\frac{1}{\lambda^4}\right)_q\\
               &={}_{2}F_{1}\left(\left.\begin{array}{cc}
                T^{3t}&T^{t}\\
		{} &T^{2t}
               \end{array}\right|\frac{1}{\lambda^4}\right)_q
\end{align*}
\end{proof}

\section{Proof of Theorem \ref{thm:K3PointCount}}
\label{sec:DworkSurfaces}
  
\begin{proof}[Proof of Theorem \ref{thm:K3PointCount}]
 In Section \ref{sec:Koblitz} we found that
 $$\#X_{\lambda}^4(\mathbb F_q)= N_q(0)+ S_{(0,0,0,0)}+S_{(0,0,2,2)}+S_{(0,1,1,2)}.$$
 Combining the results of Propositions \ref{prop:N0}, \ref{prop:0000}, and \ref{prop:0022} gives us
 \begin{align*}
  \#X_{\lambda}^4(\mathbb F_q)&=q^2+7q+1  + \frac1q\sum_{i=1}^3g(T^{it})^4 +12qT^t(-1)-\frac{1}{q}\sum_{i=1}^3g\left(T^{it}\right)^4\\
  & \hspace{1in}+q^2{}_{3}F_{2}\left(\left.\begin{array}{ccc}
                T^t&T^{2t}&T^{3t}\\
		{} &\epsilon&\epsilon
               \end{array}\right|\frac{1}{\lambda^4}\right)_q \\ 
                &\hspace{1in}-6q+3q^2\binom{T^{3t}}{T^t}{}_{2}F_{1}\left(\left.\begin{array}{cc}
                T^{3t}&T^{t}\\
		{} &T^{2t}
               \end{array}\right|\frac{1}{\lambda^4}\right)_q\\
               &\hspace{1in}+ 12qT^t(-1)\left(T^{2t}(1-\lambda^4)-1\right) \\\\
        &=\frac{q^3-1}{q-1}+12qT^t(-1)T^{2t}(1-\lambda^4)\\
               &\hspace{.2in}+q^2{}_{3}F_{2}\left(\left.\begin{array}{ccc}
                T^t&T^{2t}&T^{3t}\\
{} &\epsilon&\epsilon
               \end{array}\right|\frac{1}{\lambda^4}\right)_q+3q^2\binom{T^{3t}}{T^t}{}_{2}F_{1}\left(\left.\begin{array}{cc}
                T^{3t}&T^{t}\\
{} &T^{2t}
               \end{array}\right|\frac{1}{\lambda^4}\right)_q.\\
 \end{align*}
 In the case where $\lambda^4=1$, this, combined with Corollaries \ref{cor:0112b} and \ref{cor:0022}, gives us
\begin{align*}
  \#X_{\lambda}^4(\mathbb F_q)&=\frac{q^3-1}{q-1}+3qT^t(-1)+q^2{}_{3}F_{2}\left(\left.\begin{array}{ccc}
                T^t&T^{2t}&T^{3t}\\
		{} &\epsilon&\epsilon
               \end{array}\right|1\right)_q.
\end{align*}

\end{proof}

\section{Proof of Theorems \ref{thm:K3PointCountnGn} and \ref{thm:K3PointCountnGn1}}\label{sec:K3PointCountnGn}
In this section we prove our two $p-$adic hypergeometric point count formulas.

\begin{proof}[Proof of Theorem \ref{thm:K3PointCountnGn}]
 Let $N_{p}^A(\lambda)$ denote the number of points on the Dwork K3 surface in $\mathbb A^4(\mathbb F_p)$. Then 
\begin{equation}\label{eqn:padicPointCont}
 \#X_{\lambda}^4(\mathbb F_p)=\frac{N_{p}^A(\lambda)-1}{p-1}.
\end{equation}
Letting $f(\overline x)=x_1^4+x_2^4+x_3^4+x_4^4-4\lambda x_1x_2x_3x_4$ we can write
 \begin{align*}
 pN_{p}^A(\lambda)&=p^4+\sum_{z\in\mathbb F_p^*}\sum_{x_1,x_2,x_3,x_4}\theta(zf(\overline x))\\
 &=p^4+\sum_{z\in\mathbb F_p^*}\sum_{x_i\not=0}\theta(zf(\overline x))+\sum_{z\in\mathbb F_p^*}\sum_{x_i, some =0}\theta(zf(\overline x)).
\end{align*}

We will call the first summand A and the second B. We first work to rewrite B. We can have 1, 2, 3, or 4 of the $x_i$'s equal to zero, and there are 4, 6, 4, and 1 way, respectively, for this to occur. We will call these sums $B_1, B_2, B_3,$ and $B_4$, respectively. $B_4=p-1$ because $\theta$ is an additive character. We can simplify the others using basic facts about characters and Gauss sums
\begin{align*}
 B_1  &=4\sum_{z\in\mathbb F_p^*}\sum_{x_i\not=0}\theta(zx_1^4)\theta(zx_2^4)\theta(zx_3^4)\\
  &=\frac{4}{(p-1)^3}\sum_{x_i,z\in\mathbb F_p^*}\sum_{a,b,c=0}^{p-2} g(T^{-a})T^{4a}(x_1)g(T^{-b})T^{4b}(x_2)g(T^{-c})T^{4c}(x_3)T^{a+b+c}(z)\\
  &=\frac{4}{(p-1)^3}\sum_{a,b,c=0}^{p-2}g(T^{-a})g(T^{-b})g(T^{-c})\sum_{x_1}T^{4a}(x_1)\sum_{x_2}T^{4b}(x_2)\sum_{x_3}T^{4c}(x_3)\sum_{z}T^{a+b+c}(z).
\end{align*}
This sum is non-zero only when the following congruences hold:
$$4a,4b,4c\equiv 0 \pmod{p-1},\text{ and }  a+b+c\equiv 0 \pmod{p-1}.$$
Since $p\not\equiv 1\pmod 4$, these congruences simultaneously hold only when 0 or 2 of $a,b,c$ are $\frac{p-1}{2}$ and the remaining terms are 0. In this case, each character sum is $p-1$. Note that there are 3 ways to have two of $a,b,c$ equal to zero. Thus,
\begin{align*}
 B_1&=4(p-1)\left(g(T^0)^3+3g\left(T^{(p-1)/2}\right)^2g(T^0)\right)\\
    &=4(p-1)\left(-1+3(-1p)(-1)\right)\\
    &=4(p-1)(3p-1).
\end{align*}

Similarly, 
\begin{align*}
B_2  &=6\sum_{z\in\mathbb F_p^*}\sum_{x_i\not=0}\theta(zx_1^4)\theta(zx_2^4)\\
  &=\frac{6}{(p-1)^2}\sum_{a,b=0}^{p-2}g(T^{-a})g(T^{-b})\sum_{x_1}T^{4a}(x_1)\sum_{x_2}T^{4b}(x_2)\sum_{z}T^{a+b}(z).
\end{align*}
This sum is non-zero only when the following congruences hold:
$$4a,4b\equiv 0 \pmod{p-1},\text{ and }  a+b\equiv 0 \pmod{p-1}.$$
Since $p\not\equiv 1\pmod 4$, these congruences simultaneously hold only when 0 or 2 of $a,b$ are $\frac{p-1}{2}$. Thus,
\begin{align*}
 B_2&=6(p-1)\left(g(T^0)^2+g(T^{(p-1)/2})^2\right)\\
    &=-6(p-1)(p-1).
\end{align*}

Finally, 
\begin{align*}
B_3  &=4\sum_{z\in\mathbb F_p^*}\sum_{x_1\not=0}\theta(zx_1^4)\\
  &=\frac{4}{p-1}\sum_{a=0}^{p-2}g(T^{-a})\sum_{x_1}T^{4a}(x_1)\sum_{z}T^{a}(z).
\end{align*}
This sum is non-zero only when $a=0$. Thus,
\begin{align*}
 B_2&=4(p-1)g(T^0)=-4(p-1).
\end{align*}

Putting this all together gives
\begin{align*}
B&=B_1+B_2+B_3+B_4\\
 &=4(p-1)(3p-1)-6(p-1)(p-1)-4(p-1)+(p-1)\\
 &=(p-1)(6p-1).
\end{align*}

Rewriting summand A requires more work and we will end up with $p-$adic hypergeomtric functions.
\begin{align*}
 A&=\sum_{z\in\mathbb F_p^*}\sum_{x_i\not=0}\theta(zf(\overline x))\\
  &=\sum_{z\in\mathbb F_p^*}\sum_{x_i\not=0}\theta(zx_1^4)\theta(zx_2^4)\theta(zx_3^4)\theta(zx_4^4)\theta(-4\lambda x_1x_2x_3x_4)\\
  &=\frac{1}{(p-1)^5}\sum_{i,j,k,l,m=0}^{p-2} g(T^{-i})g(T^{-j})g(T^{-k})g(T^{-l})g(T^{-m})T^{m}(-4\lambda)\sum T^{4i+m}(x_1)\\
  &\hspace{1.5in} \times \sum T^{4j+m}(x_2)\sum T^{4k+m}(x_3)\sum T^{4l+m}(x_4)\sum T^{i+j+k+l+m}(z).
\end{align*}
We consider congruences that must hold for $i,j,k,l,m$ as we did for summand B. This sum is non-zero only when the following congruences hold:
$$4i+m,4j+m, 4k+m, 4l+m\equiv 0 \pmod{p-1},\text{ and }  i+j+k+l+m\equiv 0 \pmod{p-1}.$$
There are two cases. In the first case, which we will denote by $A_1$, $i,j,k,l$ are equal in pairs. Here we need two equal to each other and the remaining two equal to that value plus $\frac{p-1}{2}$. For example, $i=j$ and $k=l=i+\frac{p-1}{2}$. In this case we have $m\equiv -4i \pmod{p-1}$. There are 3 ways for this to occur. In the second case we have $i=j=k=l$ and $m\equiv -4j \pmod{p-1}$. We will denote this case by $A_2$.\\ 

We now work to rewrite $A_1$. We start with
\begin{align*}
 A_1&=3\sum_{i=0}^{p-2} g(T^{-i})^2g(T^{-(i+(p-1)/2})^2g(T^{4i})T^{-4i}(-4\lambda).
\end{align*}
We use the Hasse-Davenport relation with $\chi^{(p-1)/2}$ and $\psi=T^{-i}$ to get
\begin{align*}
g(T^{-(i+(p-1)/2})^2&=\left(\frac{-g(T^{-2i})T^{2i}(2)g(T^0)g(T^{((p-1)/2})}{g(T^{-i})}\right)^2\\
    &=\frac{-pg(T^{-2i})^2T^{2i}(4)}{g(T^{-i})^2}.
\end{align*}
Thus,
\begin{align*}
A_1&=-3p\sum_{i=0}^{p-2} g(T^{-2i})^2g(T^{4i})T^{2i}(4)T^{-4i}(-4\lambda)\\
  &=6p-3p\sum_{i\not=0,(p-1)/2} g(T^{-2i})^2g(T^{4i})T^{2i}(4)T^{-4i}(-4\lambda)
\end{align*}
We multiply the summand by $g(T^{2i})/g(T^{2i})$ to get
\begin{align*}
A_1&=6p-3p^2\sum_{i\not=0,(p-1)/2} \frac{g(T^{-2i})g(T^{4i})T^{2i}(4)T^{-4i}(-4\lambda)}{g(T^{2i})}.
\end{align*}

As we have done in previous proofs, we let $T=\omega$, the Teichm\"uller character. Thus,
\begin{align*}
A_1&=6p-3p^2\sum_{i\not=0,(p-1)/2} \frac{g(\overline\omega^{2i})g(\overline\omega^{-4i})\overline\omega^{-2i}(4)\overline\omega^{4i}(4\lambda)}{g(\overline\omega^{-2i})}.
\end{align*}
We use the Gross-Koblitz formula to write this in terms of the $p-$adic Gamma function.
\begin{align*}
A_1&=6p+3p^2\sum_{i\not=0,(p-1)/2} \frac{(-p)^a\Gamma_p\left(\left\langle\frac{2i}{p-1}\right\rangle\right)\Gamma_p\left(\left\langle\frac{-4i}{p-1}\right\rangle\right)\overline\omega^{4i}(2\lambda)}{\Gamma_p\left(\left\langle\frac{-2i}{p-1}\right\rangle\right)},
\end{align*}
where
\begin{align*}
a&=\left\langle\frac{2i}{p-1}\right\rangle+\left\langle\frac{-4i}{p-1}\right\rangle-\left\langle\frac{-2i}{p-1}\right\rangle\\
  &=\frac{2i}{p-1}-\left\lfloor\frac{2i}{p-1}\right\rfloor+\frac{-4i}{p-1}-\left\lfloor\frac{-4i}{p-1}\right\rfloor-\frac{-2i}{p-1}+\left\lfloor\frac{-2i}{p-1}\right\rfloor\\
  &=\left\lfloor\frac{-2i}{p-1}\right\rfloor-\left\lfloor\frac{2i}{p-1}\right\rfloor-\left\lfloor\frac{-4i}{p-1}\right\rfloor.
\end{align*}
On page 232 of  \cite{McCarthy2013} we see that
$$\left\lfloor\frac{-2i}{p-1}\right\rfloor-\left\lfloor\frac{4i}{p-1}\right\rfloor= -\left\lfloor\frac34-\frac{i}{p-1}\right\rfloor-\left\lfloor\frac14-\frac{i}{p-1}\right\rfloor.$$
Furthermore,
$$\left\lfloor\frac{2i}{p-1}\right\rfloor=\begin{cases} 
      0 & \text{if } i< \frac{p-1}{2},\\
      1 & \text{if } i\geq \frac{p-1}{2},
\end{cases}$$
so that $\left\lfloor\frac{2i}{p-1}\right\rfloor=\left\lfloor\frac{i}{p-1}\right\rfloor+\left\lfloor\frac12+\frac{i}{p-1}\right\rfloor$.
Thus,
$$a=-\left\lfloor\frac34-\frac{i}{p-1}\right\rfloor-\left\lfloor\frac14-\frac{i}{p-1}\right\rfloor-\left\lfloor\frac{i}{p-1}\right\rfloor-\left\lfloor\frac12+\frac{i}{p-1}\right\rfloor.$$

We use Lemma 4.1 of \cite{McCarthy2013} to rewrite the $p-$adic Gamma functions that appear in the summand. This yields the following expression for $A_1$.
\begin{align*}
A_1&=6p+3p^2\sum_{i\not=0,(p-1)/2}(-p)^a\cdot \frac{\Gamma_p\left(\left\langle\frac{i}{p-1}\right\rangle\right)\Gamma_p\left(\left\langle\frac12+\frac{i}{p-1}\right\rangle\right)\prod_{h=0}^3\Gamma_p\left(\left\langle\frac{1+h}{4}-\frac{i}{p-1}\right\rangle\right)\overline\omega^{i}(\lambda^4)}{\prod_{h=1}^3\Gamma_p\left(\frac{h}{4}\right)\cdot\Gamma_p\left(\left\langle\frac12-\frac{i}{p-1}\right\rangle\right)\Gamma_p\left(\left\langle1-\frac{i}{p-1}\right\rangle\right)}\\
   &=6p+3p^2\sum_{i\not=0,(p-1)/2}(-p)^a\cdot \frac{\Gamma_p\left(\left\langle\frac{i}{p-1}\right\rangle\right)\Gamma_p\left(\left\langle\frac12+\frac{i}{p-1}\right\rangle\right)}{\Gamma_p\left(\frac{1}{4}\right)\Gamma_p\left(\frac{2}{4}\right)\Gamma_p\left(\frac{3}{4}\right)}\\
   &\hspace{2in}\times\Gamma_p\left(\left\langle\frac{1}{4}-\frac{i}{p-1}\right\rangle\right)\Gamma_p\left(\left\langle\frac{3}{4}-\frac{i}{p-1}\right\rangle\right)\overline\omega^{i}(\lambda^4).
\end{align*}
If $i=0$ or $(p-1)/2$, then
\begin{align*}
(-p)^a\cdot \frac{\Gamma_p\left(\left\langle\frac{i}{p-1}\right\rangle\right)\Gamma_p\left(\left\langle\frac12+\frac{i}{p-1}\right\rangle\right)\Gamma_p\left(\left\langle\frac{1}{4}-\frac{i}{p-1}\right\rangle\right)\Gamma_p\left(\left\langle\frac{3}{4}-\frac{i}{p-1}\right\rangle\right)\overline\omega^{i}(\lambda^4)}{\Gamma_p\left(\frac{1}{4}\right)\Gamma_p\left(\frac{2}{4}\right)\Gamma_p\left(\frac{3}{4}\right)}&=(-p)^{0}=1.
\end{align*}

Thus,
\begin{align*}
A_1&=6p-6p^2+3p^2\sum_{i=0}^{p-2}(-p)^a\cdot\frac{\Gamma_p\left(\left\langle\frac{1}{4}-\frac{i}{p-1}\right\rangle\right)\Gamma_p\left(\left\langle\frac{3}{4}-\frac{i}{p-1}\right\rangle\right)}{\Gamma_p\left(\left\langle\frac{-1}{4}\right\rangle\right)\Gamma_p\left(\left\langle\frac{-3}{4}\right\rangle\right)}\frac{\Gamma_p\left(\left\langle\frac{i}{p-1}\right\rangle\right)\Gamma_p\left(\left\langle\frac12+\frac{i}{p-1}\right\rangle\right)}{\Gamma_p\left(0\right)\Gamma_p\left(\left\langle\frac12\right\rangle\right)}.
\end{align*}
We recognize this sum as the following $p-$adic hypergeometric function expression
$$A_1=6p-6p^2-3(p-1)p^2{}_2G_2\left[\left.\begin{array}{ccc}
                3/4&1/4\\
		0&1/2
               \end{array}\right|{\lambda^4}\right]_p.$$

We now work to rewrite $A_2$. In this case we have
\begin{align*}
 A_2&=\sum_j g(T^{-j})^4g(T^{4j})T^{-4j}(-4\lambda).
\end{align*}
As we have done in previous proofs, we let $T=\omega$, the Teichm\"uller character.  Then
\begin{align*}
 A_2&=\sum_{j=0}^{p-2} g(\overline\omega^{j})^4g(\overline\omega^{-4j})\overline\omega^{4j}(4\lambda)\\
  &=-1+\sum_{j=1}^{p-2} g(\overline\omega^{j})^4g(\overline\omega^{-4j})\overline\omega^{4j}(4\lambda)\\
  &=-1+p\sum_{j=1}^{p-2} \frac{g(\overline\omega^{j})^3g(\overline\omega^{-4j})\overline\omega^{j}(-1)\overline\omega^{4j}(4\lambda)}{g(\overline\omega^{-j})}.
\end{align*}
We would like to rewrite the summand in terms of the $p$-adic Gamma function. To do this, we use the Gross-Koblitz formula.
\begin{align*}
 A_2&=-1-p\sum_{j=1}^{p-2} (-p)^k\frac{\Gamma_p\left(\langle\frac{j}{p-1}\rangle\right)^3\Gamma_p\left(\langle\frac{-4j}{p-1}\rangle\right)}{\Gamma_p\left(\langle\frac{-j}{p-1}\rangle\right)}\overline\omega^{j}(-1)\overline\omega^{4j}(4\lambda),
\end{align*}
where
\begin{align*}
k&=3\left\langle\frac{j}{p-1}\right\rangle+\left\langle\frac{-4j}{p-1}\right\rangle-\left\langle\frac{-j}{p-1}\right\rangle\\
 &=\frac{3j}{p-1}-3\left\lfloor\frac{j}{p-1}\right\rfloor+\frac{-4j}{p-1}-\left\lfloor\frac{-4j}{p-1}\right\rfloor-\frac{-j}{p-1}+\left\lfloor\frac{-j}{p-1}\right\rfloor\\
 &=-\left\lfloor\frac{-4j}{p-1}\right\rfloor-1\\
 &=\begin{cases} 
      0 & \text{if } 0< j\leq \frac{p-1}{4},\\
      1 & \text{if } \frac{p-1}{4}< j\leq \frac{2(p-1)}{4},\\
      2 & \text{if } \frac{2(p-1)}{4}< j\leq \frac{3(p-1)}{4},\\
      3 &  \text{if } \frac{3(p-1)}{4}< j\leq p-2.
 \end{cases}
\end{align*}
%Lemma 4.1 in \cite{McCarthy2013} tells us that
%\begin{align*}
%\Gamma_p\left(\left\langle \frac{-4j}{p-1}\right\rangle\right)\overline\omega^{4j}(4)\prod_{h=1}^{3}\Gamma_p\left(\frac{h}{4}\right)&=\prod_{h=0}^{3}\Gamma_p\left(\left\langle\frac{1+h}{4}-\frac{j}{p-1}\right\rangle\right).
%\end{align*}

We use Lemma 4.1 in \cite{McCarthy2013} to rewrite $\Gamma_p\left(\left\langle \frac{-j}{p-1}\right\rangle\right)$ and $ \Gamma_p\left(\left\langle \frac{-4j}{p-1}\right\rangle\right)$. This leads to
\begin{align*}
  A_2&=-1-p\sum_{j=1}^{p-2} (-p)^k\frac{\Gamma_p\left(\langle\frac{j}{p-1}\rangle\right)^3\prod_{h=0}^{3}\Gamma_p\left(\left\langle\frac{1+h}{4}-\frac{j}{p-1}\right\rangle\right)}{\Gamma_p\left(\langle1-\frac{j}{p-1}\rangle\right)\prod_{h=1}^{3}\Gamma_p\left(\frac{h}{4}\right)}\overline\omega^{j}(-\lambda^4)\\
 &=-1-p\sum_{j=1}^{p-2} (-p)^k\frac{\Gamma_p\left(\left\langle\frac{j}{p-1}\right\rangle\right)^3}{\Gamma_p(0)^3}\cdot\prod_{i=1}^3\frac{\Gamma_p\left(\left\langle\frac{i}{4}-\frac{j}{p-1}\right\rangle\right)}{\Gamma_p\left(\frac{i}{4}\right)}\overline\omega^{j}(-\lambda^4).
\end{align*}
This expression can be written in terms of a $p$-adic hypergeometric function.  By Definition \ref{def:padicHGF}, we have
\begin{align*}
 {}_3G_3\left[\left.\begin{array}{ccc}
                1/4&2/4&3/4\\
		0&0&0
               \end{array}\right|{\lambda^4}\right]_p &:=\frac{-1}{p-1}\sum_{j=0}^{p-2}\overline\omega^j(-\lambda^4)\prod_{i=1}^3\frac{\Gamma_p\left(\left\langle\frac{i}{4}-\frac{j}{p-1}\right\rangle\right)}{\Gamma_p\left(\frac{i}{4}\right)}\cdot\frac{\Gamma_p\left(\left\langle\frac{j}{p-1}\right\rangle\right)^3}{\Gamma_p(0)^3}\\
               &\hspace{1in}\cdot(-p)^{-\lfloor\frac14-\frac{j}{p-1}\rfloor-\lfloor\frac24-\frac{j}{p-1}\rfloor-\lfloor\frac34-\frac{j}{p-1}\rfloor-3\lfloor\frac{j}{p-1}\rfloor}.
\end{align*}
Note that when $j=0$, the summand is simply $-\frac{1}{p-1}$. Our main task now is to determine what the power of $-p$ is for other values of $j$. First note that since $0\leq j\leq p-2$, we have $\lfloor\frac{j}{p-1}\rfloor=0$. For $i=1,2,3$ we have
$$\left\lfloor\frac{i}{4}-\frac{j}{p-1}\right \rfloor=\begin{cases} 
      0 & \text{if }\frac{j}{p-1}\leq \frac{i}{4},\\
      -1 &   \text{if }\frac{j}{p-1}> \frac{i}{4}.
 \end{cases}
$$
Thus, the exponent of $-p$ is
$$-\lfloor\tfrac14-\tfrac{j}{p-1}\rfloor-\lfloor\tfrac24-\tfrac{j}{p-1}\rfloor-\lfloor\tfrac34-\tfrac{j}{p-1}\rfloor-3\lfloor\tfrac{j}{p-1}\rfloor= \begin{cases} 
      0 & \text{if } 0< j\leq \frac{p-1}{4},\\
      1 & \text{if } \frac{p-1}{4}< j\leq \frac{2(p-1)}{4},\\
      2 & \text{if } \frac{2(p-1)}{4}< j\leq \frac{3(p-1)}{4},\\
      3 &  \text{if } \frac{3(p-1)}{4}< j\leq p-2.
 \end{cases}
$$
Note how this coincides with the powers of $-p$ in the summand $A_2$. Thus,
\begin{align*}
 {}_3G_3\left[\left.\begin{array}{ccc}
                1/4&2/4&3/4\\
		0&0&0
               \end{array}\right|{\lambda^4}\right]_p&= \\
               &\hspace{-.2in}\frac{-1}{p-1}\left[ 1+\sum_{j=1}^{p-2} (-p)^k\frac{\Gamma_p\left(\left\langle\frac{j}{p-1}\right\rangle\right)^3}{\Gamma_p(0)^3}\prod_{i=1}^3\frac{\Gamma_p\left(\left\langle\frac{i}{4}-\frac{j}{p-1}\right\rangle\right)}{\Gamma_p\left(\frac{i}{4}\right)}\overline\omega^{j}(-\lambda^4)\right].
\end{align*}
This allows us to conclude that 
\begin{align*}
 A_2=p(p-1) {}_3G_3\left[\left.\begin{array}{ccc}
                1/4&2/4&3/4\\
		0&0&0
               \end{array}\right|{\lambda^4}\right]_p +p-1.
\end{align*}

We finish the proof by combining these results to get an expression for $ \#X_{\lambda}^4(\mathbb F_p)$. First note that
\begin{align*}
 A+B&=B+A_1+A_2\\
    &=(p-1)(6p-1)+6p-6p^2-3(p-1)p^2{}_2G_2\left[\left.\begin{array}{ccc}
                3/4&1/4\\
		0&1/2
               \end{array}\right|{\lambda^4}\right]_p\\
               &\hspace{1in}+p(p-1){}_3G_3\left[\left.\begin{array}{ccc}
                1/4&2/4&3/4\\
		0&0&0
               \end{array}\right|{\lambda^4}\right]_p +p-1\\
    &=-3(p-1)p^2{}_2G_2\left[\left.\begin{array}{ccc}
                3/4&1/4\\
		0&1/2
               \end{array}\right|{\lambda^4}\right]_p+p(p-1){}_3G_3\left[\left.\begin{array}{ccc}
                1/4&2/4&3/4\\
		0&0&0
               \end{array}\right|{\lambda^4}\right]_p.
\end{align*}
We now substitute this into Equation \ref{eqn:padicPointCont}.
\begin{align*}
\#X_{\lambda}^4(\mathbb F_p)&=\frac{N_{p}^A(\lambda)-1}{p-1}\\
			 &=\frac{\frac1p(p^4+A+B)-1}{p-1}\\
			 &=\frac{p^3-1}{p-1}-3p\hspace{.02in}{}_2G_2\left[\left.\begin{array}{ccc}
                3/4&1/4\\
		0&1/2
               \end{array}\right|{\lambda^4}\right]_p+{}_3G_3\left[\left.\begin{array}{ccc}
                1/4&2/4&3/4\\
		0&0&0
               \end{array}\right|{\lambda^4}\right]_p
\end{align*}

\end{proof}

\begin{proof}[Proof of Theorem \ref{thm:K3PointCountnGn1}]
 
 We start with the point count result over $\mathbb F_p$ of Theorem \ref{thm:K3PointCount}
  \begin{align*}
 \#X_{\lambda}^4(\mathbb F_p)&=p^2+p+12pT^t(-1)T^{2t}(1-\lambda^4)+1\\
               &\hspace{.2in}+p^2{}_{3}F_{2}\left(\left.\begin{array}{ccc}
                T^t&T^{2t}&T^{3t}\\
{} &\epsilon&\epsilon
               \end{array}\right|\frac{1}{\lambda^4}\right)_p+3p^2\binom{T^{3t}}{T^t}{}_{2}F_{1}\left(\left.\begin{array}{cc}
                T^{3t}&T^{t}\\
{} &T^{2t}
               \end{array}\right|\frac{1}{\lambda^4}\right)_p
\end{align*}

We use transformation properties from \cite{McCarthy2013, McCarthy2012c} to rewrite the two finite field hypergeometric expressions in terms of McCarthy's $p-$adic hypergeometric function.  In \cite{McCarthy2012c}, McCarthy defines a new finite field hypergeometric function, normalized to satisfy transformation properties based on summation properties of Gauss sums. He also gives the relationship between McCarthy's hypergeometric function and Greene's hypergeometric function in \cite[Prop. 2.5]{McCarthy2012c}.
 \begin{align}\label{eqn:GreeneMcCarthyTrans}
  {}_{n+1}F_{n}\left(\left.\begin{array}{cccc}
                A_0,&A_1,&\ldots,&A_n\\
		{} &B_1,&\ldots,&B_n
               \end{array}\right|x\right)^M_p&=  \prod_{i=1}^n\binom{A_i}{B_i}^{-1} {}_{n+1}F_{n}\left(\left.\begin{array}{cccc}
                A_0,&A_1,&\ldots,&A_n\\
		{} &B_1,&\ldots,&B_n
               \end{array}\right|x\right)_p.
 \end{align}
Furthermore, in \cite[Lemma 3.3]{McCarthy2013} McCarthy gives the following relationship between his finite field hypergeometric function and his $p$-adic hypergeometric function. For a fixed odd prime $p$, let $A_i,B_k$ be given by $\overline\omega^{a_i(p-1)}$ and $\overline\omega^{b_k(p-1)}$ respectively. Then
\begin{align}\label{eqn:McCarthyTransnGn}
 {}_{n+1}F_{n}\left(\left.\begin{array}{cccc}
                A_0,&A_1,&\ldots,&A_n\\
		{} &B_1,&\ldots,&B_n
               \end{array}\right|x\right)^M_p&={}_{n+1}G_{n+1}\left[\left.\begin{array}{cccc}
                a_0&a_1&\ldots&a_n\\
		0&b_1&\ldots&b_n
               \end{array}\right|x^{-1}\right]_p
\end{align}    
We use these two properties to rewrite the hypergeometric function expressions in Theorem \ref{thm:K3PointCount}. We start with the ${}_3F_2$ term. Equation 2.12 in \cite{Greene} states that for a character $A\in\widehat{\mathbb F_q^{\times}}$ we have
$$\binom{A}{\epsilon}=-\frac1p.$$
We use this and Equation \ref{eqn:GreeneMcCarthyTrans} above to write
\begin{align*}
p^2{}_{3}F_{2}\left(\left.\begin{array}{ccc}
                T^t&T^{2t}&T^{3t}\\
		{} &\epsilon&\epsilon
               \end{array}\right|\frac{1}{\lambda^4}\right)_p &= p^2\binom{T^{2t}}{\epsilon}\binom{T^{3t}}{\epsilon} {}_{3}F_{2}\left(\left.\begin{array}{ccc}
                T^t&T^{2t}&T^{3t}\\
		{} &\epsilon&\epsilon
               \end{array}\right|\frac{1}{\lambda^4}\right)_p^M\\
               &= p^2\left(-\frac1p\right)\left(-\frac1p\right){}_{3}F_{2}\left(\left.\begin{array}{ccc}
                T^t&T^{2t}&T^{3t}\\
		{} &\epsilon&\epsilon
               \end{array}\right|\frac{1}{\lambda^4}\right)_p^M\\
               &= {}_{3}F_{2}\left(\left.\begin{array}{ccc}
                T^t&T^{2t}&T^{3t}\\
		{} &\epsilon&\epsilon
               \end{array}\right|\frac{1}{\lambda^4}\right)_p^M.
\end{align*}
Equation \ref{eqn:McCarthyTransnGn} then tells us that this is equal to 
$${}_3G_3\left[\left.\begin{array}{ccc}
                1/4&2/4&3/4\\
		0&0&0
               \end{array}\right|{\lambda^4}\right]_p.$$
               
We now work to rewrite the ${}_3F_2$ term. First note that
\begin{align*}
\binom{T^{3t}}{T^t}\binom{T^t}{T^{2t}}&=\frac{g(T^{3t})g(T^{3t})T^t(-1)}{g(T^{2t}p}\cdot\frac{g(T^{t})g(T^{2t})T^{2t}(-1)}{g(T^{3t}p}\\
                                      &=\frac{g(T^{3t})g(T^{t})T^t(-1)}{p^2}\\
                                     % &=\frac{pT^t(-1)\cdot T^t(-1)}{p^2}\\
                                      &=\frac1p.
\end{align*}
We use this and Equations \ref{eqn:GreeneMcCarthyTrans} and \ref{eqn:McCarthyTransnGn} to write 
\begin{align*}
3p^2\binom{T^{3t}}{T^t}{}_{2}F_{1}\left(\left.\begin{array}{cc}
                T^{3t}&T^{t}\\
		{} &T^{2t}
               \end{array}\right|\frac{1}{\lambda^4}\right)_p&=3p^2\binom{T^{3t}}{T^t}\binom{T^t}{T^{2t}}{}_{2}F_{1}\left(\left.\begin{array}{cc}
                T^{3t}&T^{t}\\
		{} &T^{2t}
               \end{array}\right|\frac{1}{\lambda^4}\right)_p^M\\
               &=3p{}_{2}F_{1}\left(\left.\begin{array}{cc}
                T^{3t}&T^{t}\\
		{} &T^{2t}
               \end{array}\right|\frac{1}{\lambda^4}\right)_p^M\\
               &=3p{}_2G_2\left[\left.\begin{array}{cc}
                3/4&1/4\\
		0&2/4
               \end{array}\right|\lambda^4\right]_p.
\end{align*}

We have rewritten each finite field hypergeometric expression as a $p$-adic hypergeometric expression, and so we have proved the desired result.

\end{proof}

%\begin{cor}
 %When $p\equiv 1 \pmod 4$ and $\lambda^4=1$ the point count for the Dwork K3 surface
 %$$ x_1^4+x_2^4+x_3^4+x_4^4=4\lambda x_1x_2x_3x_4$$
 %is given by
%\begin{align*}
%\#X_{\lambda}^4(\mathbb F_p)&=p^2+p+1+{}_3G_3\left[\left.\begin{array}{ccc}
%                1/4&2/4&3/4\\
%		0&0&0
%               \end{array}\right|1\right]_p.
%\end{align*}

%\end{cor}

\section{Dwork K3 Surface Period Integrals}\label{sec:DworkPeriod}
In this section we give a formula for certain period integrals associated to Dwork K3 surfaces. The periods we are interested in are obtained by choosing dual bases of the space of holomorphic differentials and the space of cycles ${H^2(X_\lambda^4,\mathscr O)}$ and integrating the differentials over each cycle. Note that since K3 surfaces have genus $g=1$, the dimension of both spaces is 1 and, so, we get a single period integral. The natural choice for a basis of differentials is the nowhere vanishing holomorphic $2$-form.  \\

The following theorem tells us that the period obtained in this way is a solution to a hypergeometric differential equation.

\begin{prop}\label{prop:K3picardfuchs}\cite[Section 3.2]{Nagura1995}
 The Picard-Fuchs equation for the Dwork K3 surface is 
 \begin{equation}\label{eqn:K3picardfuchs}
\left(\vartheta^3-z(\vartheta+1/4)(\vartheta+2/4)(\vartheta+3/4)\right)\pi=0,
 \end{equation}
where $\vartheta=z\frac{d}{dz}$ and $z=\lambda^{-4}$.
\end{prop}

The proposition below gives us a classical hypergeometric series expression for the period. The result is not original, but we prove it to review the procedure.
\begin{prop}\label{prop:K3period}
 The solution to Equation \ref{eqn:K3picardfuchs} that is bounded near $z=0$ is given by
 \begin{equation}\label{eqn:K3period}
  \pi = {}_{3}F_{2}\left(\left.\begin{array}{ccc}
                1/4&2/4&3/4\\
		{} &1&1
               \end{array}\right|\frac{1}{\lambda^4}\right).
 \end{equation}

\end{prop}

\begin{proof}
 Define $D:=\vartheta^3-z(\vartheta+1/4)(\vartheta+2/4)(\vartheta+3/4)$. Let $f(z)=\sum_m a_mz^m$ be a solution to $Df=0$, normalized so that $f(0)=1$, and let $f':=\frac{df}{dz}$. We see that
 \begin{align*}
  \left(\vartheta^3\right)f&=\left(z\frac{d}{dz}\right)^3f\\
			   &=\left(z\frac{d}{dz}\right)^2(zf')\\
			   &=\left(z\frac{d}{dz}\right)(zf'+z^2f'')\\
			   %&=zf'+z^2f''+2z^2f''+z^3f'''\\
			   &=z^3f'''+3z^2f''+zf'.
 \end{align*}
Furthermore,
\begin{align*}
 (\vartheta+1/4)(\vartheta+2/4)(\vartheta+3/4)f&=(\vartheta+1/4)(\vartheta+2/4)(zf'+\tfrac34f)\\
					       %&=(\vartheta+1/4)\left(\tfrac24zf'+\tfrac{6}{16}f+zf'+z^2f''+\tfrac34zf'\right)\\
					       &=(\vartheta+1/4)(z^2f''+\tfrac94zf'+\tfrac38f)\\
					       %&=\tfrac14z^2f''+\tfrac{9}{16}zf'+\tfrac{3}{32}f+2z^2f''+z^3f'''+\tfrac94zf'+\tfrac94z^2f''+\tfrac38zf'\\
					       &=z^3f'''+\tfrac92z^2f''+\tfrac{51}{16}zf'+\tfrac{3}{32}f.
\end{align*}
Hence,
\begin{align*}
 Df&=z^3f'''+3z^2f''+zf'-z(z^3f'''+\tfrac92z^2f''+\tfrac{51}{16}zf'+\tfrac{3}{32}f)\\
   &=z^3(1-z)f'''+z^2(3-\tfrac92z)f''+z(1-\tfrac{51}{16}z)f'-\tfrac{3}{32}zf.
\end{align*}
Taking derivatives of $f$ gives us
\begin{align*}
 f'&=\sum_m (m+1)a_{m+1}z^m,\\
 f''&=\sum_m (m+2)(m+1)a_{m+2}z^m,\\
 f'''&=\sum_m (m+3)(m+2)(m+1)a_{m+3}z^m.
\end{align*}
We substitute these into the equation $Df=0$ to get
\begin{align*}
\sum_m \left(z^3(1-z)(m+3)(m+2)(m+1)a_{m+3}+z^2(3-\tfrac92z)(m+2)(m+1)a_{m+2}\right.&{}\\
\left.+z(1-\tfrac{51}{16}z)(m+1)a_{m+1}-\tfrac{3}{32}za_m\right)z^m&=0. 
\end{align*}
From the expression on the left side of the above equation we identify the coefficients of $z^m$ to rewrite the equation as
\begin{align*}
 %\sum_m \left(\left((m+1)+3m(m+1)+(m-1)(m)(m+1)\right)a_{m+1} - \left( \tfrac{3}{32}+\tfrac{51}{16}m+\tfrac92m(m-1)+m(m-1)(m-2)\right)a_m\right)z^m\\
 \sum_m \left((m+1)^3a_{m+1}-\tfrac{1}{64}(4m+1)(4m+2)(4m+3)a_m\right)z^m =0
\end{align*}
This gives the relationship
\begin{align*}
 a_{m+1}&=\frac{(4m+1)(4m+2)(4m+3)}{64(m+1)^3}a_m\\
        &=\frac{(m+1/4)(m+2/4)(m+3/4)}{(m+1)^3}a_m.
\end{align*}
Given that $a_0=1$ (so that $f(0)=1$) we get that 
\begin{align*}
 a_m=\frac{\left(\frac14\right)_m\left(\frac24\right)_m\left(\frac34\right)_m}{(1)_m(1)_m m!}.
\end{align*}
We now use Equation \ref{eqn:classicalHGF} to express the function $f$ as the hypergeometric function
\begin{equation}
 f= {}_{3}F_{2}\left(\left.\begin{array}{ccc}
                1/4&2/4&3/4\\
		{} &1&1
               \end{array}\right|z\right).
\end{equation}
We conclude the proof by setting $\pi=f$ and by recalling that $z=\frac{1}{\lambda^4}$. 

\end{proof}

\subsection{Proof of Theorem \ref{thm:K3PeriodTrace}}\label{sec:PeriodTrace}

Our final result for Dwork K3 surfaces relates the trace of Frobenius to the period associated to the surface.

\begin{proof}[Proof of Theorem \ref{thm:K3PeriodTrace}]
We denote the trace of Frobenius of the Dwork K3 surface over $\mathbb F_p$ by $a_{X^4_{\lambda}}(p)$. Recall that the trace of Frobenius and the point count of Dwork K3 surface are related in the following way
\begin{equation}
a_{X^4_{\lambda}}(p)=  \#X_{\lambda}^4(\mathbb F_p)-p^2-1.
\end{equation}
See \cite[Theorem 27.1]{Manin1986} for a proof of this. Thus, we have that the trace is given by
\begin{align*}
 a_{X^4_{\lambda}}(p)&=p+ 12pT^t(-1)T^{2t}(1-\lambda^4)\\
 &\hspace{1in}+p^2{}_{3}F_{2}\left(\left.\begin{array}{ccc}
                T^t&T^{2t}&T^{3t}\\
		{} &\epsilon&\epsilon
               \end{array}\right|\frac{1}{\lambda^4}\right)_p
               +3p^2\binom{T^{3t}}{T^t}{}_{2}F_{1}\left(\left.\begin{array}{cc}
                T^{3t}&T^{t}\\
		{} &T^{2t}
               \end{array}\right|\frac{1}{\lambda^4}\right)_p.
\end{align*}
We will show that, modulo $p$, the only term that remains is the ${}_3F_2$ hypergeometric function. We have the following lemma for the ${}_2F_1$ term.

\begin{lem}\label{lem:K3PeriodTrace}

\begin{align*}
 3p^2\binom{T^{3t}}{T^t}{}_{2}F_{1}\left(\left.\begin{array}{cc}
                T^{3t}&T^{t}\\
		{} &T^{2t}
               \end{array}\right|x\right)_p \equiv 0 \pmod p.
\end{align*}

\end{lem}
\begin{proof}
%This is clearly true when  $x \equiv 0 \pmod p$, so we assume that  $x\not \equiv 0 \pmod p$. 
Recall from the proof of Proposition \ref{prop:0022} that
\begin{align*}
 p^2\binom{T^{3t}}{T^t}{}_{2}F_{1}\left(\left.\begin{array}{cc}
                T^{3t}&T^{t}\\
		{} &T^{2t}
               \end{array}\right|x\right)_p&=g(T^{2t})g(T^{3t})^2T^t(-1){}_{2}F_{1}\left(\left.\begin{array}{cc}
                T^{3t}&T^{t}\\
		{} &T^{2t}
               \end{array}\right|x\right)_p.  
\end{align*}

We rewrite the hypergeometric function as 
\begin{align*}
{}_{2}F_{1}\left(\left.\begin{array}{cc}
                T^{3t}&T^{t}\\
		{} &T^{2t}
               \end{array}\right|x\right)_p&= \frac{p}{p-1}\sum_{\chi}\binom{T^{3t}\chi}{\chi}\binom{T^{t}\chi}{ T^{2t}\chi}\chi(x)\\
               &=\frac{1}{p(p-1)}\sum_{\chi} J(T^{3t}\chi,\overline\chi)J(T^{t}\chi,T^{2t}\overline\chi) \chi(-x)\\
               &=\frac{1}{p(p-1)}\sum_{\chi} \frac{g(T^{3t}\chi)g(\overline\chi)}{g(T^{3t})}\frac{g(T^{t}\chi)g(T^{2t}\overline\chi)}{g(T^{3t})} \chi(-x).
\end{align*}
As we have done in previous proofs, we let $T=\overline\omega$ and $\chi=\overline\omega^{-j}$, where $\omega$ is the Teichm\"uller character and rewrite our Gauss sum expression to get
\begin{align*}
 &p^2\binom{T^{3t}}{T^t}{}_{2}F_{1}\left(\left.\begin{array}{cc}
                T^{3t}&T^{t}\\
		{} &T^{2t}
               \end{array}\right|x\right)_p\\
               &\hspace{1in}=\frac{g(\overline\omega^{2t})g(\overline\omega^{3t})^2\overline\omega^t(-1)}{p(p-1)}\sum_{j=0}^{p-2} \frac{g(\overline\omega^{3t-j})g(\overline\omega^j)}{g(\overline\omega^{3t})}\frac{g(\overline\omega^{t-j})g(\overline\omega^{2t+j})}{g(\overline\omega^{3t})}\overline\omega^{-j}(-x)\\
               &\hspace{1in}=\frac{g(\overline\omega^{2t})\overline\omega^t(-1)}{p(p-1)}\sum_{j=0}^{p-2} g(\overline\omega^{3t-j})g(\overline\omega^j)g(\overline\omega^{t-j})g(\overline\omega^{2t+j})\overline\omega^{-j}(-x).
\end{align*}
Using the Gross-Koblitz formula from Theorem \ref{thm:GrossKoblitz} allows us to rewrite this as
\begin{align*}
 \frac{-\pi^{2t}\overline\omega^t(-1)}{p(p-1)}\sum_{j=0}^{p-2} \pi^{6t} \Gamma_p\left(\frac{3t-j}{p-1}\right) \Gamma_p\left(\frac{j}{p-1}\right) \Gamma_p\left(\frac{t-j}{p-1}\right) \Gamma_p\left(\frac{2t+j}{p-1}\right) \Gamma_p\left(\frac{2t}{p-1}\right)\overline\omega^{-j}(-x).
\end{align*}
Since $\pi^{2t}\cdot\pi^{6t}=\pi^{8t}=p^2$, the above expression is equal to
\begin{align*}
 \frac{-p\overline\omega^t(-1)}{p-1}\sum_{j=0}^{p-2} \Gamma_p\left(\frac{3t-j}{p-1}\right) \Gamma_p\left(\frac{j}{p-1}\right) \Gamma_p\left(\frac{t-j}{p-1}\right) \Gamma_p\left(\frac{2t+j}{p-1}\right) \Gamma_p\left(\frac{2t}{p-1}\right)\overline\omega^{-j}(-x).
\end{align*}
Recall that $\omega(x)\equiv x \pmod p$ for all $x$ in $\{0,\ldots,p-1\}$ and $p-1\equiv -1 \pmod p$. Thus, the above sum is congruent modulo $p$ to the expression 
\begin{align*}
  p(-1)^t\sum_{j=0}^{p-2} \Gamma_p\left(\frac34+j\right) \Gamma_p\left(-j\right) \Gamma_p\left(\frac14 +j\right) \Gamma_p\left(\frac24-j\right) \Gamma_p\left(\frac24\right)(-x)^{j}.
\end{align*}
This expression is congruent to 0 modulo $p$ since $\Gamma_p: \mathbb Z_p \rightarrow \mathbb Z_p^*$.  Hence,
\begin{align*}
 p^2\binom{T^{3t}}{T^t}{}_{2}F_{1}\left(\left.\begin{array}{cc}
                T^{3t}&T^{t}\\
		{} &T^{2t}
               \end{array}\right|x\right)_p &\equiv 0 \pmod p.
\end{align*}

\end{proof}

We can now write that 
\begin{align*}
a_{X^4_{\lambda}}(p)&\equiv p^2{}_{3}F_{2}\left(\left.\begin{array}{ccc}
                T^t&T^{2t}&T^{3t}\\
		{} &\epsilon&\epsilon
               \end{array}\right|\frac{1}{\lambda^4}\right)_p \pmod p.
\end{align*}

Recall that in Corollary \ref{cor:3F2congruence} we showed that
\begin{align*}
 {}_{3}F_{2}\left(\left.\begin{array}{ccc}
                \tfrac{1}{4}&\tfrac{2}{4}&\tfrac{3}{4}\\
		{} &1&1
               \end{array}\right|x\right)_{\text{tr}(p)} \equiv  p^2\hspace{.05in}{}_{3}F_{2}\left(\left.\begin{array}{ccc}
                T^{t}&T^{2t}&T^{3t}\\
		{} &\epsilon&\epsilon
               \end{array}\right|x\right)_p \pmod p.
\end{align*}

This truncated series is equal to
\begin{align*}
\sum_{j=0}^{t}\frac{\left(\tfrac14\right)_j\left(\tfrac24\right)_j\left(\tfrac34\right)_j}{j!^3}x^j.
\end{align*}

Note that for the truncated series and the classical series we have the congruence
\begin{align*}
 \sum_{j=0}^{t}\frac{\left(\tfrac14\right)_j\left(\tfrac24\right)_j\left(\tfrac34\right)_j}{j!^3}x^j \equiv \sum_{j=0}^{\infty}\frac{\left(\tfrac14\right)_j\left(\tfrac24\right)_j\left(\tfrac34\right)_j}{j!^3}x^j \pmod p,
\end{align*}
since the terms with $j>t$ are congruent to 0 modulo $p$. Thus, we have 
\begin{align*}
a_{X^4_{\lambda}}(p)&\equiv {}_{3}F_{2}\left(\left.\begin{array}{ccc}
                1/4&2/4&3/4\\
		{} &1&1
               \end{array}\right|\frac{1}{\lambda^4}\right) \pmod p.
\end{align*}

The expression on the right side of the above congruence is the classical hypergeometric series that we used in Equation \ref{eqn:K3period} to express the period of the K3 surface, so we have proved the result.
\end{proof}

\section{Higher Dimensional Dwork Hypersurfaces} \label{sec:DworkHypersurface}
We have made partial progress on similar point count and period results for higher dimensional Dwork hypersurfaces
 $$X_{\lambda}^d: \hspace{.1in} x_1^d+x_2^d+\ldots+x_d^d=d\lambda x_1x_2\cdots x_d.$$

\subsection{Point Count for Prime Powers $q\equiv 1 \pmod d$}

We begin with a point count formula that holds for prime powers $q\equiv 1\pmod d$. This is a partial breakdown of Koblitz's point count formula of Section \ref{sec:Koblitz} and demonstrates that we should expect to be able to develop hypergeometric point count formulas for a large class of varieties.

\begin{thm}\label{thm:DworkHyp}
 Let $q\equiv 1\pmod d$, $t=\frac{q-1}{d}$, and $T$ be a generator for $\widehat{\mathbb F_q^{\times}}$. The number of points over $\mathbb F_q$ on the Dwork hypersurface is given by
\begin{align*}
\#X_{\lambda}^d(\mathbb F_q)&=\frac{q^{d-1}-1}{q-1} + q^{d-2}{}_{d-1}F_{d-2}\left(\left.\begin{array}{cccc}
                T^t&T^{2t}&\ldots &T^{(d-1)t}\\
		{} &\epsilon&\ldots&\epsilon
               \end{array}\right|\frac{1}{\lambda^d}\right)_q\\
               &\hspace{.5in} +\frac1q\sum_{\overline w\in W^{**}} \prod_i g(T^{w_it})+\frac1{q-1}\sum_{\overline w\not=\overline0}\sum_{j=0}^{p-2}\frac{\prod_{i=1}^dg\left(T^{w_it+j}\right)}{g(T^{dj})}T^{dj}(d\lambda),
\end{align*}
where the $d-$tuples $\overline w=(w_1,\ldots,w_d)$ satisfy $0\leq w_i<d$ and $\sum w_i\equiv 0 \pmod d$, and  $W^{**}$ is the subset with $w_i\not=0$ and $w_i$ not all equal.
               
\end{thm}
 
 \begin{rem*}
Note that in the case where $d=4$, i.e. for Dwork K3 surfaces, the point count formula had a main ${}_3F_2$ hypergeometric term, a ${}_2F_1$ term, and a sum of multiples of $q$. It seems likely that we will have a similar breakdown of terms in the higher dimensional case.
 \end{rem*}

 \begin{proof}[Proof of Theorem \ref{thm:DworkHyp}]
This proof follows our work in proving Theorem \ref{thm:K3PointCount}. 
Let $W$ be the set of all $d$-tuples $w=(w_1,\ldots,w_d)$ satisfying $0\leq w_i<d$ and $\sum w_i\equiv 0 \pmod d$. We denote the points on the diagonal hypersurface
$$x_1^d+\ldots+x_d^d=0$$
by $N_q(0):=\sum N_{q}(0,w)$, where
$$
 N_{q}(0,w)=
 \begin{cases}
  0 &\text{if some but not all } w_i=0,\\
  \frac{q^{d-1}-1}{q-1} &\text{if all } w_i=0,\\
  -\frac1q J\left(T^{\tfrac{w_1}{d}},\ldots,T^{\tfrac{w_d}{d}}\right) &\text{if all } w_i\not=0.\\
 \end{cases}
$$
Koblitz's formula in this general case is as follows.
$$\#X_{\lambda}^d(\mathbb F_q)=N_q(0)+\frac1{q-1}\sum\frac{\prod_{i=1}^dg\left(T^{w_it+j}\right)}{g(T^{4j})}T^{dj}(d\lambda)$$
where the sum is taken over $j\in\{0,\ldots,q-2\}$ and $w\in W$.\\

Letting $W^*$ be set of all $d-$tuples where no $w_i=0$, we can write
\begin{align*}
N_{q}(0,w)&=\frac{q^{d-1}-1}{q-1} + \frac1q\sum_{w\in W^*} \prod_i g(T^{w_it}).
\end{align*}

As in the Section \ref{sec:Koblitz} we consider cosets of $W$ with respect to the equivalence relation $\sim$ on $W$ defined by $w\sim w'$ if $w-w'$ is a multiple of $(1,\ldots,1)$. In the case where $d=4$ we had three cosets and their permutations. For general $d$, we should expect many more cosets. Regardless of the value of $d$, however, one of the cosets will be the zero element $w=(0,\ldots,0)^1$. We show that the summand associated to this coset can be expressed as a finite field hypergeometric function.\\

When $w=(0,\ldots,0)$ we have
\begin{align*}
 S_{(0,0,0,0)}&=\frac{1}{q-1}\sum_{j=0}^{q-2}\frac{g\left(T^{j}\right)^d}{g\left(T^{dj}\right)}T^{dj}(d\lambda)
\end{align*}

If $t\mid j$, then 
\begin{align*}
 \frac{g\left(T^{j}\right)^d}{g\left(T^{dj}\right)}T^{dj}(d\lambda)&=-g\left(T^{j}\right)^d.
\end{align*}
Thus,
\begin{align*}
 S_{(0,0,0,0)}&=-\frac1{q-1}\sum_{i=1}^{d-1}g\left(T^{it}\right)^d + \frac{1}{q-1}\sum_{j=0, t\nmid j}^{q-2}\frac{g\left(T^{j}\right)^d}{g\left(T^{dj}\right)}T^{dj}(d\lambda)\\
 	      &= -\frac1{q-1}\sum_{i=1}^{d-1}g\left(T^{it}\right)^d + \frac{1}{q-1}\sum_{j=0, t\nmid j}^{q-2}\frac{g\left(T^{j}\right)^dg\left(T^{-dj}\right)}{T^{dj}(-1)q}T^{dj}(d\lambda)\\
	      &= -\frac1{q-1}\sum_{i=1}^{d-1}g\left(T^{it}\right)^d + \frac{1}{q(q-1)}\sum_{j=0, t\nmid j}^{q-2}g\left(T^{j}\right)^dg\left(T^{-dj}\right)T^{dj}(-d\lambda).
\end{align*}
Note that if $t\mid j$ then
\begin{align*}
 g\left(T^{j}\right)^dg\left(T^{-dj}\right)T^{dj}(-d\lambda)&=-g\left(T^{j}\right)^d.
\end{align*}
Hence,
\begin{align*}
 S_{(0,0,0,0)}&=-\frac1{q-1}\sum_{i=1}^{d-1}g\left(T^{it}\right)^d +\frac{1}{q(q-1)}\sum_{i=1}^{d-1}g(T^{it})^d+ \frac{1}{q(q-1)}\sum_{j=0}^{q-2}g\left(T^{j}\right)^dg\left(T^{-dj}\right)T^{dj}(-d\lambda)\\
 &=-\frac{1}{q}\sum_{i=1}^{d-1}g\left(T^{it}\right)^d+ \frac{1}{q(q-1)}\sum_{j=0}^{q-2}g\left(T^{j}\right)^dg\left(T^{-dj}\right)T^{dj}(-d\lambda).
\end{align*}
We would like to express this is as a finite field hypergeometric function. Recall that the ${}_{d-1}F_{d-2}$ hypergeometric function is given by
\begin{align*}
 {}_{d-1}F_{d-2}&\left(\left.\begin{array}{cccc}
                T^t&T^{2t}&\ldots &T^{(d-1)t}\\
		{} &\epsilon&\ldots&\epsilon
               \end{array}\right|\frac{1}{\lambda^d}\right)_q\\
               &= \frac{q}{q-1}\sum_{\chi}\binom{T^t\chi}{\chi}\binom{T^{2t}\chi}{\epsilon\chi}\cdots\binom{T^{(d-1)t}\chi}{\epsilon\chi}\chi(1/{\lambda^d}).
               \end{align*}
We rewrite this so that it is in terms of Gauss sums
\begin{align*}
               &=\frac{q}{q-1}\sum_{\chi} \left(\frac{\chi(-1)}{q}\right)^{d-1} J(T^{t}\chi,\overline\chi)\cdots J(T^{(d-1)t}\chi,\overline\chi)\overline\chi({\lambda^d})\\
               &=\frac{1}{q^{d-2}(q-1)}\sum_{\chi} \chi(-1)^{d-1} \frac{\prod_{i=1}^{d-1} g(T^{it}\chi)}{\prod_{i=1}^{d-1} g(T^{it})}\overline\chi({\lambda^d})\\
\end{align*}
Use the Hasse-Davenport formula of Theorem \ref{thm:HasseDavenport} to get 
\begin{align*}
 {}_{d-1}F_{d-2}\left(\left.\begin{array}{cccc}
                T^t&T^{2t}&\ldots &T^{(d-1)t}\\
		{} &\epsilon&\ldots&\epsilon
               \end{array}\right|\frac{1}{\lambda^d}\right)_q
               &= \frac{1}{q^{d-2}(q-1)}\sum_{\chi} \chi(-1)^{d-1} \frac{g(\chi^d)g(\overline\chi)^{d-1}}{\chi^d(d)g(\chi)}\overline\chi({\lambda^d})\\
               %&= \frac{1}{q^{d-2}(q-1)}\sum_{\chi} \chi(-1)^{d-1} \frac{g(\chi^d)g(\overline\chi)^{d}}{\overline\chi(-1)q}\overline\chi^d(d\lambda)\\
               %&= \frac{1}{q^{d-1}(q-1)}\sum_{\chi} \chi(-1)^{d} g(\chi^d)g(\overline\chi)^{d}\overline\chi^d(d\lambda)\\
               &= \frac{1}{q^{d-1}(q-1)}\sum_{\chi}g(\chi^d)g(\overline\chi)^{d}\overline\chi^d(-d\lambda)\\
               &= \frac{1}{q^{d-1}(q-1)}\sum_{j=0}^{q-2}g(T^{-dj})g(T^{j})^{d}T^{dj}(-d\lambda).
\end{align*}

Thus,
\begin{align*}
S_{(0,0,0,0)}&=-\frac{1}{q}\sum_{i=1}^{d-1}g\left(T^{it}\right)^d+ q^{d-2}{}_{d-1}F_{d-2}\left(\left.\begin{array}{cccc}
                T^t&T^{2t}&\ldots &T^{(d-1)t}\\
		{} &\epsilon&\ldots&\epsilon
               \end{array}\right|\frac{1}{\lambda^d}\right)_q
\end{align*}
We now combine our results to get
\begin{align*}
 \#X_{\lambda}^d(\mathbb F_q)&=\frac{q^{d-1}-1}{q-1} + q^{d-2}{}_{d-1}F_{d-2}\left(\left.\begin{array}{cccc}
                T^t&T^{2t}&\ldots &T^{(d-1)t}\\
		{} &\epsilon&\ldots&\epsilon
               \end{array}\right|\frac{1}{\lambda^d}\right)_q\\
               &\hspace{.5in} +\frac1q\sum_{w\in W^*} \prod_i g(T^{w_it})-\frac{1}{q}\sum_{i=1}^{d-1}g\left(T^{it}\right)^d +\frac1{q-1}\sum_{\overline w\not=\overline0}\sum_{j=0}^{p-2}\frac{\prod_{i=1}^dg\left(T^{w_it+j}\right)}{g(T^{dj})}T^{dj}(d\lambda)\\
               &=\frac{q^{d-1}-1}{q-1} + q^{d-2}{}_{d-1}F_{d-2}\left(\left.\begin{array}{cccc}
                T^t&T^{2t}&\ldots &T^{(d-1)t}\\
		{} &\epsilon&\ldots&\epsilon
               \end{array}\right|\frac{1}{\lambda^d}\right)_q\\
               &\hspace{.5in} +\frac1q\sum_{w\in W^{**}} \prod_i g(T^{w_it})+\frac1{q-1}\sum_{\overline w\not=\overline0}\sum_{j=0}^{p-2}\frac{\prod_{i=1}^dg\left(T^{w_it+j}\right)}{g(T^{dj})}T^{dj}(d\lambda),
\end{align*}
where $W^{**}$ is the set of $d-$tuples with $w_i\not=0$ and $w_i$ not all equal.
\end{proof}

\subsection{Point Count for Primes $p\not\equiv 1 \pmod d$}
In this section we discuss a conjecture for a point count formula that is written in terms of McCarthy's $p$-adic hypergeometric function for primes $p\not\equiv 1 \pmod d$. 

\begin{conjec}\label{conjec:DworkHypersurfacePoint}
Let $d$ be an odd prime and $p$ a prime number such that $p\not\equiv 1 \pmod d$. The number of points over $\mathbb F_p$ on the Dwork hypersurface is given by
\begin{align*}
\#X_{\lambda}^d(\mathbb F_p)&=\frac{p^{d-1}-1}{p-1}+{}_{d-1}G_{d-1}\left[\left.\begin{array}{cccc}
                1/d&2/d&\ldots&(d-1)/d\\
		0&0&\ldots&0
               \end{array}\right|{\lambda^d}\right]_p
\end{align*}

\end{conjec}

\begin{rem*}
Currently our conjecture is somewhat limited, only applying to Dwork hypersurfaces with $d$ a prime. When $d$ is not prime we have found that the number of terms to consider and simplify grows rather large. This is because we have to consider various congruences as we did in the proofs of Theorems \ref{thm:K3PointCountnGn} and \ref{thm:DworkHyp}, and the number of solutions to these is unwieldy when $d$ is not prime. \\

After the initial submission of this paper to the Arxiv, Barman, Rahman, and Saikia have demonstrated the validity of this conjecture in \cite{Barman2015}.

% For example, if $d=6$ then we need to consider solutions to
% 
% $$a_1, 6a_2, \ldots 6a_6 \equiv 0 \pmod{p-1},\text{ and }  a_1+\ldots+a_6 \equiv 0 \pmod{p-1}.$$
% 
% Though $p\not\equiv 1 \pmod 6$, we still have $p\equiv 1 \pmod 2$.... but actually not $p\equiv 1 \pmod 3$ since then we'd have $p\equiv 1 \pmod6$.\\
% 
% Actually, maybe it's not so bad when $d$ is a product of two distinct primes, one of them being $2$?
%  
\end{rem*}

\begin{proof}[Partial proof of Conjecture \ref{conjec:DworkHypersurfacePoint}]

 The start of this proof is identical to the start of the proof of Theorem \ref{thm:DworkHyp} and, so, we start at the point when we consider the two sums A and B. Starting with B we have 
\begin{align*}
 B&=\frac{1}{(q-1)^d}\sum_{a_i=0}^{p-2}g(T^{-a_1})\cdots g(T^{-a_d})\\
  &\hspace{1in} \times\sum_{x_1}T^{da_1}(x_1)\cdots \sum_{x_d}T^{da_d}(x_d)\sum_{z}T^{a_1+\ldots a_d}(z).
\end{align*}
This sum is non-zero only when the following congruences hold:
$$da_1,\ldots , da_d\equiv 0 \pmod{p-1},\text{ and }  \sum a_i\equiv 0 \pmod{p-1}.$$
Since $p\not\equiv 1\pmod d$ and $d$ is prime, these congruences simultaneously hold only when $a_1,\ldots, a_d=0$. Hence,
$$B=-(p-1).$$

We now work to rewrite $A$. 
\begin{align*}
 A&=\frac{1}{(q-1)^{d+1}}\sum_{a_i=0}^{p-2}g(T^{-a_1})\cdots g(T^{-a_{d+1}})T^{d+1}(-d\lambda )\\
  &\hspace{1in} \times\sum_{x_1}T^{da_1+a_{d+1}}(x_1)\cdots \sum_{x_d}T^{da_d+a_{d+1}}(x_d)\sum_{z}T^{a_1+\ldots a_{d+1}}(z).
\end{align*}
We consider congruences that must hold for the $a_i$. This sum is non-zero only when the following congruences hold:
$$da_1+a_{d+1},\ldots, da_d+a_{d+1}\equiv 0 \pmod{p-1},\text{ and }  a_1+\ldots+a_{d+1}\equiv 0 \pmod{p-1}.$$
As in the proof of Theorem \ref{thm:DworkHyp} we first consider  having the $a_i$ not all equal. Here we would have $a_i=\frac{j_i(q-1)}{d}$, where $0\leq j_i\leq d-1$, $\sum{j_i}\equiv 0 \pmod d$, and the $a_i$ are not all identical. However, since $d$ is prime, this is not possible. Thus, we must have all of the $a_i$ being equal. Thus we have
\begin{align*}
 A_2&=\sum_{j=0}^{p-2} g(T^{-j})^dg(T^{dj})T^{-dj}(-d\lambda).
\end{align*}

We expect that this term can be expressed as a $p$-adic hypergeometric function of the form that we saw in Theorem \ref{thm:K3PointCountnGn}. Our conjecture is that we have
\begin{align*}
 A_2&=(-1)^dp(p-1){}_{d-1}G_{d-1}\left[\left.\begin{array}{cccc}
                1/d&2/d&\ldots&(d-1)/d\\
		0&0&\ldots&0
               \end{array}\right|{\lambda^d}\right]_p
\end{align*}
plus a term to cancel with B. Assuming this is true, we now write a formula for the point count.
\begin{align*}
\#X_{\lambda}^d(\mathbb F_p)&=\frac{\frac1p(p^d+A-B)-1}{p-1}\\
	       &=\frac{p^{d-1}-1}{p-1}+\frac{\frac1p(A_1+A_2-B)}{p-1}\\
%	       &=\frac{p^{d-1}-1}{p-1}+\frac{\frac1p(A_1-B)}{p-1}+{}_{d-1}G_{d-1}\left[\left.\begin{array}{cccc}
 %               1/d&2/d&\ldots&(d-1)/d\\
%		0&0&\ldots&0
 %              \end{array}\right|{\lambda^d}\right]_p\\
               &=\frac{p^{d-1}-1}{p-1}-{}_{d-1}G_{d-1}\left[\left.\begin{array}{cccc}
                1/d&2/d&\ldots&(d-1)/d\\
		0&0&\ldots&0
               \end{array}\right|{\lambda^d}\right]_p.\\
\end{align*}

\end{proof}

\subsection{Dwork Hypersurface Period Calculation}

\begin{prop}\label{prop:Dworkpicardfuchs}\cite[Section 3.2]{Nagura1995}
 The Picard-Fuchs equation for the Dwork hypersurface
  $$X_{\lambda}^d: \hspace{.1in} x_1^d+x_2^d+\ldots+x_d^d=d\lambda x_1x_2\cdots x_d$$
is given by
 \begin{equation}\label{eqn:Dworkpicardfuchs}
\left(\vartheta^{d-1}-z(\vartheta+\tfrac{1}{d})\cdots(\vartheta+\tfrac{d-1}{d})\right)\pi=0
 \end{equation}
where $\vartheta=z\frac{d}{dz}$ and $z=\lambda^{-d}$. 
\end{prop}

The following is adapted from results in Section 46 of Rainville's text \cite{Rainville1960}.

\begin{prop}\label{prop:Dworkperiod}
 The solution to Equation \ref{eqn:Dworkpicardfuchs} in Proposition \ref{prop:Dworkpicardfuchs} that is bounded near $z=0$ is given by
 \begin{equation}\label{eqn:Dworkperiod}
  \pi = {}_{d-1}F_{d-2}\left(\left.\begin{array}{cccc}
                \tfrac1d&\tfrac2d&\cdots&\tfrac{d-1}{d}\\
		{} &1&\cdots&1
               \end{array}\right|\frac{1}{\lambda^d}\right).
 \end{equation}

\end{prop}

We saw in Theorem \ref{thm:dFdcongruence} that this is congruent (up to a sign) modulo $p$ to the matching finite field hypergeometric function that appears in the point count. This leads us to a conjecture that extends the congruence we saw in Theorem \ref{thm:K3PeriodTrace}

\begin{conjec}\label{conj:Dworkperiodpoint}
For the Dwork hypersurface
 $$X_{\lambda}^d: \hspace{.1in}  x_1^d+x_2^d+\ldots+x_d^d=d\lambda x_1x_2\ldots x_d$$
 we have that the trace of Frobenius over $\mathbb F_p$ and the period associated to the surface are congurent modulo $p$ when $p\equiv 1\pmod d$.
\end{conjec}

\begin{rem*}
The conjecture here is that, as in the Dwork K3 surface case, either the remaining Gauss sum terms in the point count formula of Theorem \ref{thm:DworkHyp} are congruent to 0 modulo $p$ or that these terms are canceled out in the trace of Frobenius -- point count relationship. This relationship becomes more complicated for higher dimensional varieties.\\ 

For example, consider the family of Dwork threefolds
 $$X^5_{\lambda}: \hspace{.1in} x_1^5+x_2^5+\ldots+x_5^5=5\lambda x_1x_2\cdots x_5.$$
 The trace of Frobenius over $\mathbb F_p$ when $p\equiv 1 \pmod 5$  is given by
$$a_{X^5_{\lambda}}(p)=p^3+25p^2-100p+1- \#X^5_{\lambda}(\mathbb F_p).$$ 
 See \cite[Section 3.1]{Meyer2005} for a proof of this. In the case where $\lambda=1$ and $p\equiv 1 \pmod 5$, Conjecture \ref{conj:Dworkperiodpoint} follows from Theorem \ref{thm:dFdcongruence}, Proposition \, and the point count work of McCarthy in \cite{McCarthy2012b}. More generally, for $\lambda\not=1$, the formula for $\#X^5_{\lambda}(\mathbb F_p)$ has a main ${}_4F_3$ hypergeometric term and several terms made up of products of Gauss sums. \\

A result relating the trace of Frobenius and the periods is expected for algebraic curves because of Manin's work in \cite{Manin}. However, there is not a result of this sort that holds generally for higher dimensional algebraic varieties.  We expect that it should be the case that the period and trace of Frobenius over $\mathbb F_p$ are congruent for a large class of varieties. In particular it would seem possible to show, at least by comparing explicit formulas, that the trace and the period are congruent when these expressions are both hypergeometric. Better yet, given that we expect there to be a congruence between these two quantities, it seems possible that it is exactly the varieties whose periods are solutions to hypergeometric differential equations that have a finite field hypergeometric point count.

\end{rem*}

\bibliographystyle{plain}
\bibliography{DworkPaper}

\end{document}